\documentclass[a4paper]{article}
\usepackage{a4wide}
\usepackage[utf8]{inputenc}
\usepackage{amsmath,amsthm,amssymb}
\usepackage{graphicx}
\usepackage{booktabs}
\usepackage{tikz}%for figures in numerical section

\newcommand \argmax  {\operatornamewithlimits{argmax}}
\newcommand \sign  {\operatorname{sign}} %signum funktion
\newcommand \Sign  {\operatorname{Sign}} %mengenwertige signum funktion
\DeclareMathOperator{\id}{id} %Lieber so, da \operatorname manchmal zu
\newcommand \N   {\mathbb{N}}

\newcommand \R   {\mathbb{R}}

\newcommand   \eps \varepsilon
\renewcommand \phi \varphi
\newcommand \induces       \leadsto
\newcommand \contradiction \lightning % benötigt packet stmaryrd
\newcommand{\bigO}{\mathcal{O}}

\newcommand \ds            \displaystyle
\newcommand \lbar          \overline
\newcommand \mat[1]        {\begin{pmatrix} #1 \end{pmatrix}}
\newcommand \set[2]        {\{#1:#2\}}

\newcommand{\norm}[1]{\|#1\|}
\newcommand{\scp}[2]{\langle #1,#2\rangle}

\newcounter{punkteTmp}
        { \begin{list}{}{\leftmargin=0em}\item[#2.] {\it\;#1\,}\setcounter{punkteTmp}{#3} \; } %punkte zwischenspeichern
        { \hfill (\thepunkteTmp Punkte) \end{list} }

\newtheoremstyle{thm}%      name
  {\baselineskip}%          Space above
  {\baselineskip}%          Space below
  {\itshape}%               Body font
  {}%                       Indent amount (empty = no indent, \parindent = para indent)
  {\bfseries}%              Thm head font
  {.}%                      Punctuation after thm head
  {.5em}%                   Space after thm head: " " = normal interword space;
  {}%                       Thm head spec (can be left empty, meaning `normal')

\newtheoremstyle{others}%   name
  {\baselineskip}%          Space above
  {\baselineskip}%          Space below
  {\upshape}%               Body font
  {}%                       Indent amount (empty = no indent, \parindent = para indent)
  {\bfseries}%              Thm head font
  {.}%                      Punctuation after thm head
  {.5em}%                   Space after thm head: " " = normal interword space;
  {}%                       Thm head spec (can be left empty, meaning `normal')

\theoremstyle{thm}

\ifdefined \theorem {}
\else{
        \newtheorem  {theorem}{Theorem}[section]

        \newtheorem* {satz*}                {Satz}
        
        \newtheorem  {corollary}[theorem]   {Corollar}
        \newtheorem* {corollary*}           {Corollar}
        \newtheorem  {proposition}[theorem] {Proposition}
        \newtheorem  {lemma}[theorem]       {Lemma}
        \newtheorem*  {lemma*}       {Lemma}
        
        \newtheorem* {Behauptung*}          {Behauptung}

        \theoremstyle{others}
        \newtheorem  {remark}[theorem]      {Remark}
        \newtheorem* {remark*}              {Remark}
        \newtheorem  {definition}[theorem]  {Definition}
        \newtheorem* {definition*}          {Definition}

        \newtheorem  {example}[theorem]     {Example}
        \newtheorem* {example*}             {Example}
        
        \newtheorem* {notation*}            {Bezeichnung}
        
        \newtheorem* {Einfuhrung*}          {Einf"uhrung}
}
\fi

\title{Elastic-Net Regularization: Error estimates and Active Set
  Methods} \author{Bangti Jin\thanks{Center for Industrial Mathematics, University of Bremen, D--28334 Bremen, Germany
  (\texttt{btjin,schiffi@math.uni-bremen.de})} \and Dirk A. Lorenz\thanks{Institute for
    Analysis and Algebra, TU Braunschweig, D-38092 Braunschweig,
    Germany (\texttt{d.lorenz@tu-braunschweig.de})} \and Stefan
  Schiffler\footnotemark[1]
  }

\begin{document}
\maketitle

\begin{abstract}
This paper investigates theoretical properties and efficient
numerical algorithms for the so-called elastic-net regularization
originating from statistics, which enforces simultaneously $\ell^1$
and $\ell^2$ regularization. The stability of the minimizer and its
consistency are studied, and convergence rates for both \textit{a
priori} and \textit{a posteriori} parameter choice rules are
established. Two iterative numerical algorithms of active
set type are proposed, and their convergence properties are
discussed. Numerical results are presented to illustrate the
features of the functional and algorithms.
\end{abstract}

\section{Introduction}

In recent years, minimization problems involving so-called
sparsity constraints have gained considerable interest. Sparsity has
been found as a powerful tool and recognized as an important
structure in many disciplines, e.g.~geophysical problems
\cite{Taylor1979,Levy1981}, imaging science
\cite{figueiredo2003emalgorithm}, statistics \cite{Tibshirani1996}
and signal processing \cite{chen1998basispursuit,Candes2006}. The
setting is often as following: Let $\mathcal{H}_1$ and
$\mathcal{H}_2$ be two Hilbert spaces and let $\mathcal{H}_1$ be
equipped with an orthonormal basis
$\{\phi_i\in\mathcal{H}_1: i\in\N\}$ (or an overcomplete
dictionary). Then, for given linear and continuous operator
$K:\mathcal{H}_1\rightarrow\mathcal{H}_2$, data
$y^\delta\in\mathcal{H}_2$ and regularization parameter $\alpha>0$,
we seek the minimizer of the functional
\begin{equation*}
\Psi(x)=\frac{1}{2}\|Kx-y^\delta\|^2+\alpha\sum_i|\langle
x,\phi_i\rangle|.
\end{equation*}
Here $y^\delta$ is an observational version of the exact data
$y^\dagger$ and satisfies an estimate of the form
$\|y^\delta-y^\dagger\|\leq \delta$. With the help of the basis
expansion, the problem can be reformulated as
\begin{equation}\label{eqn:inv}
\min_{x\in\ell^2}\Psi(x)\quad \text{with}\quad \Psi(x)=\frac{1}{2}\|Kx-y^\delta\|^2+\alpha\|x\|_{\ell^1},
\end{equation}
by abusing the notation $x$ for the sequence of expansion
coefficients $\{x_i:=\langle x,\phi_i\rangle\}$ and 
$K$ for the operator $\{x_i\}\mapsto K\sum_ix_i\phi_i$ mapping from
$\ell^2$ to $\mathcal{H}_2$.

Because of its central importance in inverse problems and signal processing, 
the efficient minimization of the functional $\Psi$ has
received much attention, and a wide variety of numerical algorithms,
e.g.~iterated thresholding/shrinkage
\cite{daubechies2003iteratethresh,bredies2008harditer}, gradient
projection \cite{figueiredo2007gradproj,Wright2009}, fixed point
continuation \cite{ElaineT.Hale.etal:2008}, semismooth Newton method
(SSN) \cite{griesse2008ssnsparsity} and feature sign search (FSS)
\cite{lee2006featuresignsearch}, have been proposed. Both SSN and FSS
are of active set type, and have delivered favorable performance
compared to the above-mentioned first-order methods. However, they
often require inverting potentially ill-conditioned
operators, and thus lead to numerical problems. One possible remedy
is to regularize the inversion, e.g.~by Tikhonov regularization. On
the other hand, recent
studies~\cite{lorenz2008reglp,grasmair2008sparseregularization} show
the regularizing property of the functional $\Psi$ and under
suitable source conditions also the convergence rate of its
minimizer $x_\alpha^\delta$ to the true solution $x^\dagger$ of the
form
\begin{equation*}
\|x_\alpha^\delta - x^\dagger\|_{\ell^2} = \bigO(\delta).
\end{equation*}
However, the involved constant may be astronomically large. In other
words, the ill-posed problem has been turned into a well-posed but
ill-conditioned one, and this is in accordance with inverting
ill-conditioned operators. In this paper we propose to address both
issues by Tikhonov regularization, i.e.~considering a functional of
the form
\begin{eqnarray*}
    \Phi_{\alpha,\beta}(x) = \frac{1}{2}\|Kx-y^\delta\|^2+\alpha\|x\|_{\ell^1} +
    \frac{\beta}{2}\|x\|_{\ell^2}^2.
\end{eqnarray*}
We will show that this functional leads to more stable active-set
algorithms and provides improved error estimates.  We note that it
also arises by Moreau-Yosida regularization of  the Fenchel dual of
the functional $\Psi$.

The functional $\Phi_{\alpha,\beta}$ is also used in statistics under
the name elastic-net regularization \cite{Zou.etal:2008}. It is
motivated by the following observation: The functional $\Psi$
delivers undesirable results for problems where there are highly
correlated features and we need to identify all relevant ones,
e.g.~microarray data analysis, in that it tends to select only
one feature out of the relevant group instead of all
relevant features of the group \cite{Zou.etal:2008}, i.e. it fails
to identify the group structure. Zou and Hastie \cite{Zou.etal:2008}
proposed introducing an extra $\ell^2$ regularization term, i.e.~the
functional $\Phi_{\alpha,\beta}$, in the hope of retrieving
correctly the whole relevant group, and numerically confirmed the
desired property of the functional for both simulation studies and
real-data applications. For further statistical motivations we refer
to reference \cite{Zou.etal:2008}. Quite recently, De Mol \textit{et
al.} \cite{DeMol.etal:2009} showed some interesting theoretical
properties of the functional $\Phi_{\alpha,\beta}$, but their focus
is fundamentally different from ours: Their main concern is on its
statistical properties in the framework of learning theory and an
algorithm of iterated shrinkage type, whereas ours is within the
framework of classical regularization theory and algorithms of
active set type.

The rest of the paper is organized as follows. In
Section~\ref{sec:prop_elastic_net} we investigate theoretical
properties, e.g. stability and consistency of the minimizers of the
elastic-net functional. In particular, the convergence rates for
both \textit{a priori} and \textit{a posteriori} regularization
parameter choice rules are established under suitable source
conditions. In Section~\ref{sec:algorithms}, we propose two active
set algorithms, i.e.~the RSSN and RFSS, for efficiently minimizing the
functional $\Phi_{\alpha,\beta}$, and discuss their convergence
properties. In Section~\ref{sec:experiments}, numerical results are
presented to illustrate the salient features of the algorithms.

\section{Properties of elastic-net regularization}
\label{sec:prop_elastic_net} In this section we investigate the
stability and regularizing properties of elastic-net regularization.
Both \textit{a priori} and \textit{a posteriori} choice rules for
choosing the regularization parameters are considered. We shall 
denote the minimizer of the functional $\Phi_{\alpha,\beta}$ by 
$x_{\alpha,\beta}^\delta$ below, and occasionally
suppress the superscript $\delta$ for notational simplicity.
Observe that for every $\beta>0$, the functional $\Phi_{\alpha,
\beta}$ is strictly convex, and thus admits a unique minimizer.

\subsection{Stability of the minimizers $x_{\alpha,\beta}^\delta$}
\label{subsec:stability}

\begin{theorem}\label{thm:stability}
For the minimizer $x_{\alpha,\beta}^\delta$ with $\alpha,\beta>0$
there holds
\begin{equation*}
\lim_{(\alpha_n,\beta_n)\rightarrow(\alpha,\beta)}x_{\alpha_n,\beta_n}^\delta=x_{\alpha,\beta}^\delta.
\end{equation*}
\end{theorem}
\begin{proof}
The minimizing property of $x^n\equiv x_{\alpha_n,\beta_n}^\delta$
implies that the sequences
$\{\|Kx^n-y^\delta\|\}$, $\{\|x^n\|_{\ell^1}\}$, and $\{\|x^n\|_{\ell^2}\}$ are
uniformly bounded. In particular there exists a
subsequence of $\{x^n\}_n$, also denoted by $\{x^n\}_n$ converging
weakly to some $x^\ast\in\ell^2$.

By the weak continuity of $K$ and weak lower-semicontinuity of
norms, we have
\begin{equation}\label{eqn:wlsc}
   \|Kx^\ast-y^\delta\|\leq\liminf_{n\rightarrow\infty}\|Kx^n-y^\delta\|,\quad\|x^\ast\|_{\ell^1}\leq
   \liminf_{n\rightarrow\infty}\|x^n\|_{\ell^1}\quad \mbox{and}\quad
   \|x^\ast\|_{\ell^2}\leq\liminf_{n\rightarrow\infty}\|x^n\|_{\ell^2}.
\end{equation}
Consequently, we have
\begin{eqnarray*}
\Phi_{\alpha,\beta}(x^\ast)&=&\frac{1}{2}\|Kx^\ast-y^\delta\|^2+\alpha\|x^\ast\|_{\ell^1}+\frac{\beta}{2}\|x^\ast\|_{\ell^2}^2\\
&\leq&\frac{1}{2}\liminf_{n\rightarrow\infty}\|Kx^n-y\|^2+\liminf_{n\rightarrow\infty}\alpha_n\|x^n\|_{\ell^1}
+\liminf_{n\rightarrow\infty}\frac{\beta_n}{2}\|x^n\|_{\ell^2}^2\\
&\leq&\liminf_{n\rightarrow\infty}\left\{\frac{1}{2}\|Kx^n-y\|^2+\alpha_n\|x^n\|_{\ell^1}+\frac{\beta_n}{2}\|x^n\|_{\ell^2}^2\right\}\\
&=&\liminf_{n\rightarrow\infty}\Phi_{\alpha_n,\beta_n}(x^n).
\end{eqnarray*}

Next we show that $\Phi_{\alpha,\beta}(x_{\alpha,\beta}^\delta)\geq
\limsup_{n\rightarrow\infty}\Phi_{\alpha_n,\beta_n}(x^n)$. To this
end, we observe
\begin{eqnarray*}
\limsup_{n\rightarrow\infty}\Phi_{\alpha_n,\beta_n}(x^n)&\leq&\limsup_{n\rightarrow\infty}\Phi_{\alpha_n,\beta_n}(x_{\alpha,\beta}^\delta)\\
&=&\lim_{n\rightarrow\infty}\Phi_{\alpha_n,\beta_n}(x_{\alpha,\beta}^\delta)=\Phi_{\alpha,\beta}(x_{\alpha,\beta}^\delta)
\end{eqnarray*}
by the minimizing property of $x^n$. Consequently
\begin{eqnarray*}
\limsup_{n\rightarrow\infty}\Phi_{\alpha_n,\beta_n}(x^n)\leq\Phi_{\alpha,\beta}(x_{\alpha,\beta}^\delta)\leq
\Phi_{\alpha,\beta}(x^\ast)\leq\liminf_{n\rightarrow\infty}\Phi_{\alpha_n,\beta_n}(x^n).
\end{eqnarray*}
Therefore, $x^\ast$ is a minimizer of $\Phi_{\alpha,\beta}$, and the
uniqueness of its minimizer implies
$x^\ast=x_{\alpha,\beta}^\delta$. Since every subsequence has a
weakly convergent subsequence to $x_{\alpha,\beta}^\delta$, the
whole sequence $\{x^n\}_n$ converges weakly to
$x_{\alpha,\beta}^\delta$. Next we show that the functional value
$\|x^n\|_{\ell^2}\rightarrow\|x_{\alpha,\beta}^\delta\|_{\ell^2}$,
for which it suffices to show that
\begin{equation*}
\limsup_{n\rightarrow\infty}\|x^n\|_{\ell^2}\leq\|x_{\alpha,\beta}^\delta\|_{\ell^2}.
\end{equation*}
Assume that this does not hold. Then there exists a constant $c$
such that $c:=\limsup_{n\rightarrow\infty} \|x^n\|_{\ell^2}^2>
\|x_{\alpha,\beta}^\delta\|_{\ell^2}^2$, and a subsequence of
$\{x^n\}_n$, denoted by $\{x^n\}_n$ again, such that
\begin{equation*}
x^n\rightarrow x_{\alpha,\beta}^\delta\ \mbox{weakly}
\quad\mbox{and}\quad \|x^n\|_{\ell^2}^2\rightarrow c.
\end{equation*}
By the continuity of $\Phi_{\alpha,\beta}(x_{\alpha,\beta}^\delta)$
in $(\alpha,\beta)$, we have
\begin{eqnarray*}
\lim_{n\rightarrow\infty}\left\{\frac{1}{2}\|Kx^n-y^\delta\|^2+\alpha_n\|x^n\|_{\ell^1}\right\}&=&
\Phi_{\alpha,\beta}(x_{\alpha,\beta}^\delta)-\lim_{n\rightarrow\infty}\frac{\beta_n}{2}\|x\|_{\ell^2}^2\\
&=&\frac{1}{2}\|Kx_{\alpha,\beta}^\delta-y^\delta\|^2+\alpha\|x_{\alpha,\beta}^\delta\|_{\ell^1}
+\frac{\beta}{2}\left(\|x_{\alpha,\beta}^\delta\|_{\ell^2}^2-c\right)\\
&<&\frac{1}{2}\|Kx_{\alpha,\beta}^\delta-y^\delta\|^2+\alpha\|x_{\alpha,\beta}^\delta\|_{\ell^1}.
\end{eqnarray*}
This is in contradiction with the lower-semicontinuity result in
equation (\ref{eqn:wlsc}). Therefore we have
\begin{equation*}
\limsup_{n\rightarrow\infty}\|x^n\|_{\ell^2}\leq \|x_{\alpha,\beta}^\delta\|_{\ell^2}.
\end{equation*}
This together with equation (\ref{eqn:wlsc}) implies that
$\|x^n\|_{\ell^2}\rightarrow \|x_{\alpha,\beta}^\delta\|_{\ell^2}$,
from which and the weak convergence the desired convergence
in $\ell^2$ follows directly.
\end{proof}

The preceding theorem addresses only the case that both $\alpha$ and
$\beta$ are positive. The case of vanishing $\alpha$ and positive
$\beta$ is obviously the same as the uniqueness of the minimizer to
the functional $\Phi_{0,\beta}$ remains valid. The more interesting
case of vanishing $\beta$ will be discussed below. In general, due
to the potential lack of uniqueness for vanishing $\beta$, only
subsequential convergence can be expected. Interestingly,
whole-sequence convergence remains true under certain circumstances.
To illustrate the point, we denote by $\mathcal{S}_\alpha$ the set
of minimizers to the functional $\Phi_{\alpha,0}$. Clearly the set
$\mathcal{S}_\alpha$ is nonempty and convex as a consequence of the
convexity of the functional $\Phi_{\alpha,0}$. Moreover, denote the
minimum $\gamma\|\cdot\|_{\ell^1}+\frac{1}{2}\|\cdot\|_{\ell^2}^2$
element of the set $\mathcal{S}_\alpha$ by
$\hat{x}_{\alpha,\gamma}^\delta$. Since the functional
$\gamma\|\cdot\|_{\ell^1}+\frac{1}{2}\|\cdot\|_{\ell^2}^2$ is
strictly convex, $\hat{x}_{\alpha,\gamma}^\delta$ is unique.

\begin{proposition}
\label{lem:beta_to_zero}
Let the sequence $\{(\alpha_n,\beta_n)\}_n$ satisfy that for some
$\gamma\geq0$ and $\alpha>0$ there holds
\begin{equation*}
\lim_{n\rightarrow\infty}\beta_n=0\quad \mbox{and}\quad \lim_{n\rightarrow\infty}\frac{\alpha_n-\alpha}{\beta_n}=\gamma.
\end{equation*}
Then we have
\begin{equation*}
\lim_{(\alpha_n,\beta_n)\rightarrow(\alpha,0)}x_{\alpha,\beta}^\delta=\hat{x}_{\alpha,\gamma}^\delta.
\end{equation*}
\end{proposition}
\begin{proof}
Denote $x^n\equiv x_{\alpha_n,\beta_n}^\delta$ the unique minimizer
of $\Phi_{\alpha_n,\beta_n}$. By repeating the arguments of Theorem
\ref{thm:stability}, we derive that there exists a subsequence  of
$\{x^n\}_n$, also denoted by $\{x^n\}_n$, that converges weakly in
$\ell^2$ to some $x^\ast$, and moreover, $x^\ast$ is a minimizer of
$\Phi_{\alpha,0}(x)$, i.e.~$x^\ast\in\mathcal{S}_\alpha$.

The minimizing property of $x^n$ and
$\hat{x}_{\alpha,\gamma}^\delta$ implies
\begin{equation*}
\frac{1}{2}\|Kx^n-y^\delta\|_2^2+\alpha_n\|x^n\|_{\ell^1}+\frac{\beta_n}{2}\|x^n\|_{\ell^2}^2
\leq\frac{1}{2}\|K\hat{x}_{\alpha,\gamma}^\delta-y^\delta\|_2^2+\alpha_n\|\hat{x}_{\alpha,\gamma}^\delta\|_{\ell^1}
+\frac{\beta_n}{2}\|\hat{x}_{\alpha,\gamma}^\delta\|_{\ell^2}^2
\end{equation*}
and
\begin{equation*}
\frac{1}{2}\|K\hat{x}_{\alpha,\gamma}^\delta-y^\delta\|_2^2+\alpha\|\hat{x}_{\alpha,\gamma}^\delta\|_{\ell^1}
\leq\frac{1}{2}\|Kx^n-y^\delta\|_2^2+\alpha\|x^n\|_{\ell^1}.
\end{equation*}
Adding these two inequalities gives
\begin{equation*}\label{eqn:ineq}
(\alpha_n-\alpha)\|x^n\|_{\ell^1}+\frac{\beta_n}{2}\|x^n\|_{\ell^2}^2\leq
(\alpha_n-\alpha)\|\hat{x}_{\alpha,\gamma}^\delta\|_{\ell^1}+\frac{\beta_n}{2}\|\hat{x}_{\alpha,\gamma}^\delta\|_{\ell^2}^2.
\end{equation*}
Dividing by $\beta_n$ and taking the limit for $n\rightarrow+\infty$
yields
\begin{equation*}
\gamma\|x^\ast\|_{\ell^1}+\frac{1}{2}\|x^\ast\|_{\ell^2}^2\leq
\gamma\|\hat{x}_{\alpha,\gamma}^\delta\|_{\ell^1}+\frac{1}{2}\|\hat{x}_{\alpha,\gamma}^\delta\|_{\ell^2}^2,
\end{equation*}
by observing the assumption $\lim_{n\rightarrow\infty}
\frac{\alpha_n-\alpha}{\beta_n}=\gamma$. By the definition of the
$\gamma\|\cdot\|_{\ell^1}+\frac{1}{2}\|\cdot\|_{\ell^2}^2$-minimizing
element $\hat{x}_{\alpha,\gamma}^\delta$ and its uniqueness, we
conclude that  $x^\ast=\hat{x}_{\alpha,\gamma}^\delta$. Since every
subsequence of $\{x^n\}_n$ has a subsequence converging weakly to
$\hat{x}_{\alpha,\gamma}^\delta$, the whole sequence $\{x^n\}_n$
converges weakly.

Appealing to the arguments in Theorem \ref{thm:stability} again,
there holds $\|x^n\|_{\ell^1}\rightarrow\|\hat{x}_{\alpha,
\gamma}^\delta\|_{\ell^1}$, which together with the weak convergence
of the sequence implies
\begin{equation*}
x^n\rightarrow x\mbox{ in } \ell^1.
\end{equation*}
The lemma follows from the inequality
$\|x\|_{\ell^2}\leq\|x\|_{\ell^1}$.
\end{proof}

The next corollary is a direct consequence of the proofs of the
preceding results.
\begin{corollary}\label{cor:cont}
The functions $\Phi_{\alpha,\beta}(x_{\alpha,\beta}^\delta)$,
$\|x_{\alpha,\beta}^\delta\|_{\ell^1}$ and
$\|x_{\alpha,\beta}^\delta\|_{\ell^2}$ are continuous in
$(\alpha,\beta)$.
\end{corollary}

The next result shows the differentiability of the value function
$F(\alpha,\beta):=\Phi_{\alpha,\beta}(x_{\alpha,\beta}^\delta)$.
Differentiability plays an important role in efficient numerical
realization of some rules for choosing regularization parameters
\cite{Ito1992,Jin2009}.

\begin{theorem}
The value function $F(\alpha,\beta)$ is differentiable with respect
to $\alpha$ and $\beta$, and moreover
\begin{eqnarray*}
\frac{\partial
F}{\partial\alpha}=\|x_{\alpha,\beta}^\delta\|_{\ell^1}
\quad\mbox{and}\quad\frac{\partial F}{\partial
\beta}=\frac{1}{2}\|x_{\alpha,\beta}^\delta\|_{\ell^2}^2.
\end{eqnarray*}
\end{theorem}
\begin{proof}
For distinct $\alpha$ and $\tilde{\alpha}$, the minimizing property
of $x_{\alpha,\beta}^\delta$ and $x_{\tilde{\alpha},\beta}^\delta$
indicates
\begin{eqnarray*}
I&=&\frac{1}{2}\|Kx_{\alpha,\beta}^\delta-y^\delta\|^2+\alpha\|x_{\alpha,\beta}^\delta\|_{\ell^1}
+\frac{\beta}{2}\|x_{\alpha,\beta}^\delta\|_{\ell^2}^2-\frac{1}{2}\|Kx_{\tilde{\alpha},\beta}^\delta-y^\delta\|^2-
\alpha\|x_{\tilde{\alpha},\beta}^\delta\|_{\ell^1}-\frac{\beta}{2}\|x_{\tilde{\alpha},\beta}^\delta\|_{\ell^2}^2\leq0,\\
II&=&\frac{1}{2}\|Kx_{\tilde{\alpha},\beta}^\delta-y^\delta\|^2+\tilde{\alpha}\|x_{\tilde{\alpha},\beta}^\delta\|_{\ell^1}
+\frac{\beta}{2}\|x_{\tilde{\alpha},\beta}^\delta\|_{\ell^2}^2-\frac{1}{2}
\|Kx_{\alpha,\beta}^\delta-y^\delta\|^2-\tilde{\alpha}\|x_{\alpha,\beta}^\delta\|_{\ell^1}-\frac{\beta}{2}\|x_{\alpha,\beta}^\delta\|_{\ell^2}^2\leq0.
\end{eqnarray*}
Therefore, for $\alpha>\tilde{\alpha}$, we have
\begin{eqnarray*}
F(\alpha,\beta)-F(\tilde{\alpha},\beta)&=&\Phi_{\alpha,\beta}(x_{\alpha,\beta}^\delta)-\Phi_{\tilde{\alpha},\beta}(x_{\tilde{\alpha},\beta}^\delta)\\
&=&I+(\alpha-\tilde{\alpha})\|x_{\tilde{\alpha},\beta}^\delta\|_{\ell^1}
\leq(\alpha-\tilde{\alpha})\|x_{\tilde{\alpha},\beta}^\delta\|_{\ell^1},
\end{eqnarray*}
and
\begin{eqnarray*}
F(\alpha,\beta)-F(\tilde{\alpha},\beta)&=&\Phi_{\alpha,\beta}(x_{\alpha,\beta}^\delta)
-\Phi_{\tilde{\alpha},\beta}(x_{\tilde{\alpha},\beta}^\delta)\\
&=&-II+(\alpha-\tilde{\alpha})\|x_{\alpha,\beta}^\delta\|_{\ell^1}\geq(\alpha-\tilde{\alpha})\|x_{\alpha,\beta}^\delta\|_{\ell^1}.
\end{eqnarray*}
These two inequalities together give
\begin{equation*}
\|x_{\alpha,\beta}^\delta\|_{\ell^1}\leq\frac{F(\alpha,\beta)-F(\tilde{\alpha},\beta)}{\alpha-\tilde{\alpha}}\le
\|x_{\tilde{\alpha},\beta}^\delta\|_{\ell^1}.
\end{equation*}
Reversing the role of $\alpha$ and $\tilde{\alpha}$ yields a similar
inequality for $\alpha<\tilde{\alpha}$, which together with the
continuity result in Corollary \ref{cor:cont} implies the first
identity. The second identity can be shown analogously. The
differentiability of $F(\alpha,\beta)$ follows from the continuity
of the functions $\|x_{\alpha,\beta}^\delta\|_{\ell^1}$ and
$\frac{1}{2}\|x_{\alpha,\beta}^\delta\|_{\ell^2}^2$ in
$(\alpha,\beta)$, see Corollary \ref{cor:cont}.
\end{proof}

\subsection{Consistency and convergence rates}
\label{subsec:consistency_and_rates} In this section we shall
investigate the convergence behavior of the minimizers
$x_{\alpha,\beta}^\delta$ as the noise level $\delta$ tends to zero
for both \textit{a priori} and \textit{a posteriori} parameter
choice rules. To this end, we need the following definition of
$\phi$-minimizing solutions.
\begin{definition}
An element $x^\dagger$ is said to be a $\phi$-minimizing solution to
the inverse problem $Kx=y^\delta$ if it verifies
$Kx^\dagger=y^\dagger$ and
\begin{equation*}
\phi(x^\dagger)\leq \phi(x),\ \forall x \mbox{ with } Kx=y^\dagger.
\end{equation*}
\end{definition}

To simplify the notation, we introduce the functional
$\mathcal{R}_\eta$ defined by
\begin{equation*}
\mathcal{R}_\eta(x)=\eta\|x\|_{\ell^1}+\frac{1}{2}\|x\|_{\ell^2}^2.
\end{equation*}
We shall need the next result on the functional $\mathcal{R}_\eta$.
\begin{lemma}\label{lem:weakstr}
Assume that $\{x^n\}_n$ converges weakly to $x^\ast$ in $\ell^2$ and
$\mathcal{R}_\eta(x^n)$ converges to $\mathcal{R}_\eta(x^\ast)$.
Then $\mathcal{R}_\eta(x^n-x^\ast)$ converges to zero.
\end{lemma}
\begin{proof}
The assumption
$\mathcal{R}_\eta(x^n)\rightarrow\mathcal{R}_\eta(x^\ast)$ and
Fatou's lemma imply that
\begin{eqnarray*}
\limsup_{n}\mathcal{R}_\eta(x^n-x^\ast)&=&\limsup_n[2(\mathcal{R}_\eta(x^n)+\mathcal{R}_\eta(x^\ast))-2(\mathcal{R}_\eta(x^n)
+\mathcal{R}_\eta(x^\ast))+\mathcal{R}_\eta(x^n-x^\ast)]\\
&=&4\mathcal{R}_\eta(x^\ast)-\liminf_n\sum_i\left[2(\eta|x^n_i|
+\frac{1}{2}|x^n_i|^2+\eta|x^\ast_i|+\frac{1}{2}|x^\ast_i|)\right.\\
&&\left.-(\eta|x^n_i-x^\ast_i|+\frac{1}{2}|x^n_i-x^\ast_i|^2)\right]\\
&\leq&4\mathcal{R}_\eta(x^\ast)-\sum_i\liminf_n\left[2(\eta|x^n_i|+\frac{1}{2}|x^n_i|^2+\eta|x^\ast_i|+\frac{1}{2}|x^\ast_i|)\right.\\
&&\left.-(\eta|x^n_i-x^\ast_i|+\frac{1}{2}|x^n_i-x^\ast_i|^2)\right].
\end{eqnarray*}
By the weak convergence of $x^n$ to $x^\ast$, we have
$x^n_i\rightarrow x^\ast_i$ for all $i\in\mathbb{N}$. Therefore,
\begin{eqnarray*}
&&\sum_i\liminf_n\left[2(\eta|x^n_i|+\frac{1}{2}|x^n_i|^2+\eta|x^\ast_i|+\frac{1}{2}|x^\ast_i|)-(\eta|x^n_i-x^\ast_i|
+\frac{1}{2}|x^n_i-x^\ast_i|^2)\right]\\
&=&4\sum_i(\eta|x^\ast_i|+\frac{1}{2}|x^\ast_i|^2)=4\mathcal{R}_\eta(x^\ast).
\end{eqnarray*}
Combining the preceding inequalities we see
\begin{equation*}
\limsup_n\mathcal{R}_\eta(x^n-x^\ast)\leq
4\mathcal{R}_\eta(x^\ast)-4\mathcal{R}_\eta(x^\ast)=0,
\end{equation*}
i.e.~$\lim_{n\rightarrow\infty}\mathcal{R}_\eta(x^n-x^\ast)=0$.
\end{proof}

\begin{theorem}\label{thm:consistency}
Assume that the regularization parameters $\alpha(\delta)$ and
$\beta(\delta)$ satisfy
\begin{equation}\label{eqn:condcons1}
\alpha(\delta),\ \beta(\delta),\ \frac{\delta^2}{\alpha(\delta)},\
\frac{\delta^2}{\beta(\delta)}\rightarrow0\mbox{ as }
\delta\rightarrow0,
\end{equation}
and moreover that there exists some constant $\eta\geq0$
\begin{equation}\label{eqn:condcons2}
\lim_{\delta\rightarrow0}\frac{\alpha(\delta)}{\beta(\delta)}=\eta.
\end{equation}
Then the sequence of minimizers $\{x_{\alpha,\beta}^\delta\}_\delta$
converges to the $\eta\|\cdot\|_{\ell^1}+\frac{1}{2}
\|\cdot\|_{\ell^2}^2$-minimizing solution.
\end{theorem}
\begin{proof}
Let $x^\dagger$ be the unique $\eta\|\cdot\|_{\ell^1}+\frac{1}{2}
\|\cdot\|_{\ell^2}^2$-minimizing solution. The minimizing property
of $x_{\alpha,\beta}^\delta$ indicates
\begin{eqnarray*}
\frac{1}{2}\|Kx_{\alpha,\beta}^\delta-y^\delta\|^2+\alpha\|x_{\alpha,\beta}^\delta\|_{\ell^1}
+\frac{\beta}{2}\|x_{\alpha,\beta}^\delta\|_{\ell^2}^2
&\leq&\frac{1}{2}\|Kx^\dagger-y^\delta\|^2+\alpha\|x^\dagger\|_{\ell^1}+\frac{\beta}{2}\|x^\dagger\|_{\ell^2}^2\\
&\leq&\frac{1}{2}\delta^2+\alpha\|x^\dagger\|_{\ell^1}+\frac{\beta}{2}\|x^\dagger\|_{\ell^2}^2.
\end{eqnarray*}
By the assumptions on $\alpha(\delta)$ and $\beta(\delta)$, the
sequences $\{\|Kx_{\alpha,\beta}^\delta-y^\delta\|\}$ and
$\{\|x_{\alpha,\beta}^\delta\|_{\ell^2})\}_\delta$ are uniformly
bounded.
Therefore, there exists a subsequence of
$\{x_{\alpha,\beta}^\delta\}_\delta$, also denoted by
$\{x_{\alpha,\beta}^\delta\}_\delta$, and some $x^\ast\in\ell^2$,
such that
$x_{\alpha,\beta}^\delta\rightarrow x^\ast$ weakly.

By the weak lower semi-continuity and the triangle inequality we
derive
\begin{eqnarray*}
\|Kx^\ast-y^\dagger\|^2&\leq&2\liminf_{\delta\rightarrow0}(\|Kx_{\alpha,\beta}^\delta-y^\delta\|^2+\|y^\delta-y^\dagger\|^2)\\
&\leq&2\liminf_{\delta\rightarrow0}\left\{\delta^2+2\alpha(\delta)
\|x^\dagger\|_{\ell^1}+\beta(\delta)\|x^\dagger\|_{\ell^2}^2+\delta^2\right\}=0.
\end{eqnarray*}
Thereby we have $\|Kx^\ast-y^\dagger\|^2=0$, i.e.~$Kx^\ast=y^\dagger$. Similarly,
\begin{eqnarray}\label{eqn:consineq}
\eta\|x^\ast\|_{\ell^1}+\frac{1}{2}\|x^\ast\|_{\ell^2}^2&\leq&
\liminf_{\delta\rightarrow0}\left\{\frac{\alpha(\delta)}{\beta(\delta)}\|x_{\alpha,\beta}^\delta\|_{\ell^1}
+\frac{1}{2}\|x_{\alpha,\beta}^\delta\|_{\ell^2}^2\right\}\nonumber\\
&\leq&\liminf_{\delta\rightarrow0}\left\{\frac{\delta^2}{2\beta(\delta)}
+\frac{\alpha(\delta)}{\beta(\delta)}\|x^\dagger\|_{\ell^1}+\frac{1}{2}\|x^\dagger\|_{\ell^2}^2\right\}\nonumber\\
&=&\eta\|x^\dagger\|_{\ell^1}+\frac{1}{2}\|x^\dagger\|_{\ell^2}^2.
\end{eqnarray}
Since $x^\dagger$ is the unique $\eta\|\cdot\|_{\ell^1}+\frac{1}{2}
\|\cdot\|_{\ell^2}^2$-minimizing solution we
deduce $x^\ast=x^\dagger$. The whole sequence converges weakly by
appealing to the standard subsequence arguments. From inequality
(\ref{eqn:consineq}), we have
\begin{equation*}
\lim_{\delta\rightarrow0}\eta\|x_{\alpha,\beta}^\delta\|_{\ell^1}+\frac{1}{2}\|x_{\alpha,\beta}^\delta\|_{\ell^2}^2
=\eta\|x^\dagger\|_{\ell^1}+\frac{1}{2}\|x^\dagger\|_{\ell^2}^2.
\end{equation*}
By Lemma \ref{lem:weakstr} and the weak convergence of the sequence
$\{x_{\alpha,\beta}^\delta\}$, this identity implies that
\begin{equation*}
\lim_{\delta\rightarrow0}\|x_{\alpha,\beta}^\delta-x^\dagger\|_{\ell^2}^2
\leq\lim_{\delta\rightarrow0}2\mathcal{R}_\eta(x_{\alpha,\beta}^\delta-x^\dagger)=0.
\end{equation*}
\end{proof}

In Theorem \ref{thm:consistency}, the first set of conditions on
$\alpha(\delta)$ and $\beta(\delta)$, see equation
(\ref{eqn:condcons1}), is rather standard, whereas the other one in
(\ref{eqn:condcons2}) seems restrictive. The following question
arise naturally: Can we further relax this condition? It turns out
that it depends crucially on the structure of the set
$\mathcal{S}=\{x: Kx=y^\dagger\}$. Obviously, if the set
$\mathcal{S}$ consists of only a singleton, i.e.~$K$ is injective,
then the $\eta\|\cdot\|_{\ell^1}+\frac{1}{2}
\|\cdot\|_{\ell^2}^2$-minimizing solution is independent of $\eta$
and thus the condition can be dropped. In general, this condition
cannot be relaxed, as the following simple example shows.

\begin{example}
\label{ex:dep_on_eta}
Consider the two-dimensional example with
\begin{equation*}
K=\left[\begin{array}{cc}1&-2\\2&-4
\end{array}\right]\quad\mbox{and}\quad
y^\dagger=\left[\begin{array}{c}1\\2\end{array}\right].
\end{equation*}
The set $\mathcal{S}$ consists of elements of the form
\begin{equation*}
x(t)=\left[\begin{array}{c}1\\0
\end{array}\right]+t\left[\begin{array}{c}2\\1
\end{array}\right],\ t\in\mathbb{R},
\end{equation*}
and the $\mathcal{R}_\eta$-minimizing solution $x^\ast$ minimizes
\begin{eqnarray*}
\eta\|x\|_{\ell^1}+\frac{1}{2}\|x\|_{\ell^2}^2=\eta(|1+2t|+|t|)+\frac{1}{2}((1+2t)^2+t^2).
\end{eqnarray*}
After some algebraic manipulations, the solution $x^\ast$ is founded to be
\begin{eqnarray*}
x^\ast&=&\left\{\begin{array}{ll}
\left[\begin{array}{c}0\\-\frac{1}{2}
\end{array}\right], &\mbox{ if } \eta\geq\frac{1}{2}\\
\left[\begin{array}{c}\frac{1}{5}-\frac{2}{5}\eta\\
-\frac{2}{5}-\frac{1}{5}\eta
\end{array}\right],&\mbox{ if } \eta<\frac{1}{2}
\end{array}\right.
\end{eqnarray*}
Interestingly, there exists a critical value of $\eta^\ast$: for
$\eta\geq\eta^\ast$, the solution does not change, whereas for
$\eta<\eta^\ast$, the solution keeps on changing. In particular, the
condition $\lim_{\delta\rightarrow0}\frac{\alpha(\delta)}
{\beta(\delta)}=\eta$ is sharp in the latter case.
\end{example}

Denote the $\eta\|\cdot\|_{\ell^1}+\frac{1}{2}
\|\cdot\|_{\ell^2}^2$-minimizing solution by $x_\eta$. Since the
arguments for $x_{\alpha,\beta}^\delta$ in Section
\ref{subsec:stability} remain valid in the presence of constraints,
we have the following result.
\begin{lemma}\label{lem:cont}
For $\eta\geq0$, we have
\begin{equation*}
\lim_{\eta_n\rightarrow\eta} x_{\eta_n}=x_\eta,
\end{equation*}
where $x_{\infty}$ is taken to be the minimum-$\ell^2$ norm element
of the set $\mathcal{S}$ of $\ell^1$-minimizing solutions to the
inverse problem. Moreover, the following identity holds
\begin{equation*}
\frac{d}{d\eta}\left[\eta\|x_\eta\|_{\ell^1}+\frac{1}{2}\|x_\eta\|_{\ell^2}^2\right]=\|x_\eta\|_{\ell^1}.
\end{equation*}
\end{lemma}

We shall need the following monotonicity result on the value
functions $\|x_\eta\|_{\ell^1}$ and $\|x_\eta\|_{\ell^2}$.
\begin{lemma}\label{lem:mon}
The function $\|x_\eta\|_{\ell^1}$ is monotonically decreasing,
while $\|x_\eta\|_{\ell^2}^2$ is monotonically increasing with
respect to the parameter $\eta$ in the sense that for distinct
$\eta_1$ and $\eta_2$
\begin{equation*}
(\|x_{\eta_1}\|_{\ell^1}-\|x_{\eta_2}\|_{\ell^1})(\eta_1-\eta_2)\leq0\quad\mbox{and}\quad
(\|x_{\eta_1}\|_{\ell^2}^2-\|x_{\eta_2}\|_{\ell^2}^2)(\eta_1-\eta_2)\geq0.
\end{equation*}
\end{lemma}
\begin{proof}
Let $\eta_1,\eta_2\geq0$ be distinct. By the minimizing property of
$x_{\eta_1}$ and $x_{\eta_2}$, we have
\begin{eqnarray*}
&&\eta_1\|x_{\eta_1}\|_{\ell^1}+\frac{1}{2}\|x_{\eta_1}\|_{\ell^2}^2\leq\eta_1\|x_{\eta_2}\|_{\ell^1}+\frac{1}{2}\|x_{\eta_2}\|_{\ell^2}^2,\\
&&\eta_2\|x_{\eta_2}\|_{\ell^1}+\frac{1}{2}\|x_{\eta_2}\|_{\ell^2}^2\leq
\eta_2\|x_{\eta_1}\|_{\ell^1}+\frac{1}{2}\|x_{\eta_1}\|_{\ell^2}^2.
\end{eqnarray*}
Adding these two inequalities gives
\begin{equation*}
(\|x_{\eta_1}\|_{\ell^1}-\|x_{\eta_2}\|_{\ell^1})(\eta_1-\eta_2)\leq0,
\end{equation*}
i.e.~the function $\|x_\eta\|_{\ell^1}$ is monotonically decreasing
with respect to $\eta$. The monotonicity of $\|x_\eta\|_{\ell^2}$
follows analogously.
\end{proof}

We shall also need the next result on the local Lipschitz continuity
of $x_\eta$ in $\eta$. To this end, we denote by $\partial\phi$ the
subdifferential of a convex functional $\phi$, i.e.
$\partial\phi=\{\xi: \phi(y)-\phi(x)\geq\langle\xi,y-x\rangle\
\forall y\in\mathrm{dom}(\phi)\}$. Since $\|x\|_{\ell^2}^2$ is
continuous, we may apply the sum-rule and get
$\partial\mathcal{R}_\eta(x)= \eta\partial\|x\|_{\ell^1}+x$. Note
that the subdifferential $\partial\|x\|_{\ell^1}$ is set-valued, and
can be expressed in terms of the function $\mathrm{Sign}$ defined
componentwise by $\mathrm{Sign}(x)_k=\mathrm{sign}(x_k)$ for nonzero
$x_k$ and $\mathrm{Sign}(x)_k=[-1,1]$ otherwise, with the usual sign
function.

\begin{lemma}\label{lem:lipschitz}
The mapping $\eta\mapsto x_\eta$ is locally Lipschitz continuous in
$\eta$ for $\eta>0$.
\end{lemma}
\begin{proof}  Let $\xi_\eta$ be a subgradient of
$\norm{x_\eta}_{\ell^1}$. The minimizing property of $x_\eta$
indicates
\begin{equation*}
\left\langle\eta\xi_{\eta}+x_\eta,x_\eta-x\right\rangle\leq0,\quad\forall
x\in\mathcal{S}.
\end{equation*}
In particular, for distinct $\eta_1,\eta_2>0$, this yields
\begin{eqnarray*}
&&\left\langle\eta_1\xi_{\eta_1}+x_{\eta_1},x_{\eta_1}-x_{\eta_2}\right\rangle\leq0,\\
&&\left\langle\eta_2\xi_{\eta_2}+x_{\eta_2},x_{\eta_2}-x_{\eta_1}\right\rangle\leq0,
\end{eqnarray*}
by noting that both $x_{\eta_1},x_{\eta_2}\in\mathcal{S}$. Adding
these two inequalities together gives
\begin{equation}\label{eqn:optineq}
\left\langle\xi_{\eta_1}-\xi_{\eta_2},x_{\eta_1}-x_{\eta_2}\right\rangle
+\frac{1}{\eta_1}\left\langle{}x_{\eta_1}-x_{\eta_2},x_{\eta_1}-x_{\eta_2}\right\rangle\leq
\left(\frac{1}{\eta_1}-\frac{1}{\eta_2}\right)\left\langle
x_{\eta_2},x_{\eta_1}-x_{\eta_2}\right\rangle.
\end{equation}
Recall that the subgradient operator of a convex functional is
maximal monotone \cite{Rockafellar1970}, i.e.
\begin{equation*}
\left\langle\xi_{\eta_1}-\xi_{\eta_2},x_{\eta_1}-x_{\eta_2}\right\rangle\geq0.
\end{equation*}
Applying this inequality and the Cauchy-Schwartz inequality in
inequality (\ref{eqn:optineq}) yields
\begin{equation*}
\|x_{\eta_1}-x_{\eta_2}\|_{\ell^2}\leq\frac{\|x_{\eta_2}\|_{\ell^2}}{\eta_2}|\eta_1-\eta_2|,
\end{equation*}
which by reversing the role of $\eta_1$ and $\eta_2$ gives
\begin{equation*}
\|x_{\eta_1}-x_{\eta_2}\|_{\ell^2}\leq\min\left\{\frac{\|x_{\eta_1}\|_{\ell^2}}{\eta_1},\frac{\|x_{\eta_2}\|_{\ell^2}}{\eta_2}\right\}|\eta_1-\eta_2|.
\end{equation*}
This concludes the proof of the Lemma.
\end{proof}

By Lemma \ref{lem:mon}, the function $\|x_\eta\|_{\ell^2}$ is
monotonically increasing with respect to $\eta$ and
bounded, and thus the limits
$\lim_{\eta\rightarrow\infty}\|x_\eta\|_{\ell^2}$ and
$\lim_{\eta\rightarrow0}\|x_\eta\|_{\ell^2}$ exist, which will be
denoted by $\|x_\infty\|_{\ell^2}$ and $\|x_0\|_{\ell^2}$,
respectively.

\begin{theorem}\label{thm:nonconv}
Assume that $\|x_{\infty}\|_{\ell^2}>\|x_0\|_{\ell^2}$. Then there
exists a set $\mathcal{C}\subset(0,+\infty)$ of positive measure
such that for each $\eta\in\mathcal{C}$, the mapping
$\eta\to\|x_\eta\|_{\ell^2}$ strictly increasing.
\end{theorem}
\begin{proof}
As noted above, the function
$\|x_\eta\|_{\ell^2}$ is monotonically increasing and bounded, and thus it is of
bounded variation and almost everywhere differentiable. By
differentiation theory of functions of bounded variation
\cite{Attouch2006}, the derivative $D_\eta\|x_\eta\|_{\ell^2}$ can be
decomposed as
\begin{equation*}
D_\eta\|x_{\eta}\|_{\ell^2}=\frac{d\|x_{\eta}\|_{\ell^2}}{d\eta}+\mu_S+\mu_C,
\end{equation*}
where $\frac{d\|x_{\eta}\|_{\ell^2}}{d\eta}$, $\mu_S$ and
$\mu_C$ denote the Lebesgue regular, singular and Cantor
parts, respectively. By Lemmas \ref{lem:cont} and
\ref{lem:lipschitz}, the function $\|x_\eta\|_{\ell^2}$ is continuous and locally
Lipschitz, and thus both the singular and Cantor parts vanish.
Consequently, the following integral identity holds
\begin{equation*}
\int_{0}^\infty\frac{d\|x_{\eta}\|_{\ell^2}}{d\eta}d\eta=\|x_\infty\|_{\ell^2}-\|x_0\|_{\ell^2}.
\end{equation*}
By the monotonicity of Lemma \ref{lem:mon}, the integrand
$\frac{d\|x_{\eta}\|_{\ell^2}}{d\eta}$ is nonnegative. Therefore,
there exists a set $\mathcal{C}\subset(0,+\infty)$ of positive
measure, such that the integrand is positive,
i.e.~$\|x_\eta\|_{\ell^2}$ is strictly increasing.
\end{proof}

Theorem \ref{thm:nonconv} indicates for $\eta\in\mathcal{C}$ the
function $\eta\to\|x_\eta\|$ is strictly increasing. Therefore, the
condition $\lim_{\delta\rightarrow0}
\alpha(\delta)/\beta(\delta)=\eta$ in Theorem \ref{thm:consistency} for some $\eta$ at least cannot
be relaxed to:
$\liminf_{\delta\rightarrow0}\alpha(\delta)/\beta(\delta)=\eta_-$ and
$\limsup_{\delta\rightarrow0}\alpha(\delta)/\beta(\delta)=\eta_+$ for
some $\eta_-,\eta_+>0$ such that
$(\eta_-,\eta_+)\cap\mathcal{C}\neq\emptyset$. This partially
necessitates the condition
$\lim_{\delta\rightarrow0}\alpha(\delta)/\beta(\delta)=\eta$ for some
$\eta$ in Theorem \ref{thm:consistency}.

\begin{remark}
Many of our preceding results remain valid for far more general
regularization terms, e.g. general convex functionals.
\end{remark}

We are now in a position to discuss the convergence rates of
\textit{a priori} and \textit{a posteriori} parameter choice rules.
The foregoing discussions indicate that the condition
$\lim_{\delta\rightarrow0}
\frac{\alpha(\delta)}{\beta(\delta)}=\eta$ is often necessary for
ensuring the convergence as $\delta$ tends to zero. Therefore, we
shall assume that the ratio of $\alpha$ and $\beta$ is fixed,
i.e.~there exists an $\eta$ such that $\beta=\eta\alpha$, for the
choice rules. The next theorem shows that elastic-net regularization
behaves similar to classical Tikhonov regularization \cite{engl1996inverseproblems}
in that an analogous error estimate holds under a slightly changed
source condition.

\begin{theorem}\label{thm:convrate1}
Let $Kx^\dagger = y^\dagger$ and assume $\|y^\delta -
y^\dagger\|\leq\delta$. Moreover, let there be some $\eta>0$ such
that $x^\dagger$ fulfills the source condition
\begin{equation}
\label{eq:source_condition_elastic_net} \exists w: K^*w \in (\id +
\eta\Sign)(x^\dagger).
\end{equation}
Then it holds that the minimizer $x_{\alpha,\beta}^\delta$ of 
$\Phi_{\alpha,\beta}$
%\begin{equation}
%\label{eq:Phi_alpha_beta_delta} \Phi_{\alpha,\beta}^\delta(x) =
%\tfrac{1}{2}\|Kx - y^\delta\|^2 + \alpha\|x\|_{\ell^1} +
%\tfrac{\beta}{2}\|x\|_{\ell^2}^2
%\end{equation}
with $\alpha=\eta\beta$ fulfills
\[
\|Kx_{\alpha,\beta}^\delta - y^\delta\|\leq \delta + 2\beta\|w\|
\]
and
\[\|x_{\alpha,\beta}^\delta - x^\dagger\|_{\ell^2} \leq
\frac{\delta}{\sqrt{\beta}} + \sqrt{\beta}\|w\|.
\]
\end{theorem}
\begin{proof}
By the minimizing property of $x_{\alpha,\beta}^\delta$ there holds
\[
\tfrac{1}{2}\|Kx_{\alpha,\beta}^\delta -y^\delta\|^2 +
\alpha\|x_{\alpha,\beta}^\delta\|_{\ell^1} +
\tfrac{\beta}{2}\|x_{\alpha,\beta}^\delta\|_{\ell^2}^2 \leq
\tfrac{1}{2}\|Kx^\dagger - y^\delta\|^2 +
\alpha\|x^\dagger\|_{\ell^1}+\tfrac{\beta}{2}\|x^\dagger\|_{\ell^2}^2,
\]
which leads to
\[
\tfrac{1}{2}\|Kx_{\alpha,\beta}^\delta -y^\delta\|^2 +
\alpha(\|x_{\alpha,\beta}^\delta\|_{\ell^1}-\|x^\dagger\|_{\ell^1})+
\tfrac{\beta}{2}(\|x_{\alpha,\beta}^\delta\|_{\ell^2}^2
-\|x^\dagger\|_{\ell^2}^2 ) \leq
\tfrac{1}{2}\|Kx^\dagger-y^\delta\|^2.
\]
Using the identity $\|x_{\alpha,\beta}^\delta\|_{\ell^2}^2
-\|x^\dagger\|_{\ell^2}^2 = \|x_{\alpha,\beta}^\delta -
x^\dagger\|_{\ell^2}^2 +2\scp{x^\dagger}{x_{\alpha,\beta}^\delta -
x^\dagger}$ we get for any $\xi\in\Sign(x^\dagger)$
\begin{multline*}
\tfrac{1}{2}\|Kx_{\alpha,\beta}^\delta -y^\delta\|^2 +
\alpha(\underbrace{\|x_{\alpha,\beta}^\delta\|_{\ell^1} -
\|x^\dagger\|_{\ell^1} - \scp{\xi}{x_{\alpha,\beta}^\delta -
x^\dagger}}_{\geq 0}) + \alpha \scp{\xi}{x_{\alpha,\beta}^\delta -
x^\dagger}\\+ \tfrac{\beta}{2}(\|x_{\alpha,\beta}^\delta -
x^\dagger\|_{\ell^2}^2 + 2\scp{x^\dagger}{x_{\alpha,\beta}^\delta -
x^\dagger} ) \leq \tfrac{1}{2}\|Kx^\dagger - y^\delta\|^2.
\end{multline*}
We conclude
\[
\tfrac{1}{2}\|Kx_{\alpha,\beta}^\delta -y^\delta\|^2 + \beta
\scp{\eta\xi + x^\dagger}{x_{\alpha,\beta}^\delta - x^\dagger}+
\tfrac{\beta}{2}\|x_{\alpha,\beta}^\delta - x^\dagger\|_{\ell^2}^2
\leq \tfrac{1}{2}\|Kx^\dagger - y^\delta\|^2.
\]
Since $\xi\in\Sign(x^\dagger)$ is arbitrary, we may choose it in
such a way that the source condition
(\ref{eq:source_condition_elastic_net}), i.e. $\eta\xi +x^\dagger =
K^*w$, holds. Consequently,
\[
\tfrac{1}{2}\|Kx_{\alpha,\beta}^\delta -y^\delta\|^2 + \beta
\scp{w}{Kx_{\alpha,\beta}^\delta - y^\delta}+
\tfrac{\beta}{2}\|x_{\alpha,\beta}^\delta - x^\dagger\|_{\ell^2}^2
\leq \tfrac{1}{2}\|Kx^\dagger - y^\delta\|^2 + \beta
\scp{w}{Kx^\dagger - y^\delta}.
\]
Completing the squares on both sides by adding $\beta^2\|w\|^2/2$
leads to
\[
\tfrac{1}{2}\|Kx_{\alpha,\beta}^\delta -y^\delta + \beta w\|^2+
\tfrac{\beta}{2}\|x_{\alpha,\beta}^\delta - x^\dagger\|_{\ell^2}^2
\leq \tfrac{1}{2}\|Kx^\dagger - y^\delta + \beta w\|^2
\]
which proves the theorem.
\end{proof}

The source condition (\ref{eq:source_condition_elastic_net}) in
Theorem \ref{thm:convrate1} is equivalent to: There exists a
$w\in\mathcal{H}_2$ such that $K^\ast
w\in\partial\mathcal{R}_\eta(x^\dagger)$. It can be interpreted as
the existence of a Lagrange multiplier to the Lagrangian of a
constrained optimization problem \cite{burger2004convarreg}. Theorem
\ref{thm:convrate1}, in particular, implies that for the choice
$\beta=\mathcal{O}(\delta)$, the reconstruction
$x_{\alpha,\beta}^\delta$ achieves a convergence rate of order
$\mathcal{O}(\delta^{1/2})$.

The ultimate goal of elastic-net regularization is to retrieve a
sparse signal. Under the premise that the underlying signal
$x^\dagger$ is truly sparse, the convergence rate can be
significantly improved by using a technique recently developed by
Grasmair \textit{et al.} \cite{grasmair2008sparseregularization}. To
this end, we need the so-called \textit{finite basis injectivity}
property of the operator $K$.
\begin{definition}[\cite{bredies2008softthresholding}]\label{def:fbi}
An operator $K:\ell^2\to\mathcal{H}_2$ has the \emph{finite basis
injectivity} property, if for all finite subsets $I\subset\N$ the
operator $K|_I$ is injective, i.e.~for  all $u,v\in \ell^2$ with $Ku
= Kv$ and $u_i = v_i = 0$ for all $i \notin I$ it follows $u = v$.
\end{definition}

The next lemma will play a role in establishing an improved
convergence rate.
\begin{lemma} \label{lem:scineq}
Assume that the solution $x^\dagger$ is sparse and satisfies the
source condition (\ref{eq:source_condition_elastic_net}), and that
the operator $K$ satisfies the finite basis injectivity property.
Then there exist two positive constants $c_1$ and $c_2$ such that
\begin{equation*}
   \mathcal{R}_\eta(x)-\mathcal{R}_\eta(x^\dagger)\geq c_1\|x-x^\dagger\|_{\ell^2}-c_2\|K(x-x^\dagger)\|.
\end{equation*}
\end{lemma}
\begin{proof} Let $\xi\in\Sign(x^\dagger)$ such
that~(\ref{eq:source_condition_elastic_net}) is satisfied. Denote by
$\mathbb{I}$ the index set $\{i\in\mathbb{N}:
|\xi_i|>\frac{1}{2}\}$. Since $\xi\in\ell^2$, the set $\mathbb{I}$
is finite, and obviously, it contains the support of $x^\dagger$.
Let $\pi_{\mathbb{I}}$ and $\pi_{\mathbb{I}}^\perp$ be the natural
projections onto $\mathbb{I}$ and $\mathbb{N}\backslash\mathbb{I}$,
respectively. Then $\pi_{\mathbb{I}}x^\dagger=x^\dagger$ and
$\pi_\mathbb{I}^\perp x^\dagger=0$. By the finite basis injectivity
property of the operator $K$, we have for some constant $C$
\begin{equation*}
C\|K\pi_\mathbb{I} x\|\geq \|\pi_\mathbb{I}x\|_{\ell^2}.
\end{equation*}
Consequently,
\begin{eqnarray*}
\|x-x^\dagger\|_{\ell^2}&\leq&\|\pi_\mathbb{I}(x-x^\dagger)\|_{\ell^2}+\|\pi_\mathbb{I}^\perp
x\|_{\ell^2}\\
&\leq&C\|K(x-x^\dagger)\|+(1+C\|K\|)\|\pi_\mathbb{I}^\perp
x\|_{\ell^2}.
\end{eqnarray*}

The source condition (\ref{eq:source_condition_elastic_net}) implies that
\begin{eqnarray}\label{eqn:scineq2}
-\left\langle x^\dagger+\eta\xi,x-x^\dagger\right\rangle&=&-\left\langle K^\ast w,x-x^\dagger\right\rangle\nonumber\\
&=&-\left\langle w,K(x-x^\dagger)\right\rangle\nonumber\\
&\leq&\|w\|\|K(x-x^\dagger)\|.
\end{eqnarray}

Now let $m=\max_{i\notin\mathbb{I}}|\xi_i|\leq\frac{1}{2}$. By the
inequality $\|x\|_{\ell^2}\leq\|x\|_{\ell^1}$, we derive that
\begin{eqnarray*}
\|\pi_\mathbb{I}^\perp
x\|_{\ell^2}&\leq&\sum_{i\notin\mathbb{I}}|x_i|=2\sum_{i\notin\mathbb{I}}(1-m)|x_i|
\leq2\sum_{i\notin\mathbb{I}}\left(|x_i|-\xi_ix_i\right)\\
&=&2\sum_{i\notin\mathbb{I}}\left(|x_i|-|x_i^\dagger|-\xi_i(x_i-x_i^\dagger)\right)\\
&\leq&2\sum_{i\in\mathbb{N}}\left(|x_i|-|x_i^\dagger|-\xi_i(x_i-x_i^\dagger)\right)\\
&\leq&2\left(\|x\|_{\ell^1}-\|x^\dagger\|_{\ell^1}-\left\langle\xi,x-x^\dagger\right\rangle
+\frac{1}{2\eta}[\|x\|_{\ell^2}^2-\|x^\dagger\|_{\ell^2}^2-2\langle x^\dagger,x-x^\dagger\rangle]\right)\\
&=&2\eta^{-1}\left([\mathcal{R}_\eta(x)-\mathcal{R}_\eta(x^\dagger)]
-\langle\eta\xi+ x^\dagger,x-x^\dagger\rangle\right)\\
&\leq&2\eta^{-1}\left[\mathcal{R}_\eta(x)-\mathcal{R}_\eta(x^\dagger)+\|w\|\|K(x-x^\dagger)\|\right],
\end{eqnarray*}
where we have used the identity
$\|x\|_{\ell^2}^2-\|x^\dagger\|_{\ell^2}^2-2\langle
x^\dagger,x-x^\dagger\rangle =\|x-x^\dagger\|_{\ell^2}^2\geq0$,
inequality (\ref{eqn:scineq2}) and the fact that $x^\dagger$
vanishes outside the index set $\mathbb{I}$.

Combining above estimates gives
\begin{eqnarray*}
\|x-x^\dagger\|_{\ell^2}\leq (C+2\eta^{-1}(1+C\|K\|)\|w\|)
\|K(x-x^\dagger)\|+2\eta^{-1}(1+C\|K\|)[\mathcal{R}_\eta(x)-\mathcal{R}_\eta(x^\dagger)].
\end{eqnarray*}
This concludes the proof of the lemma.
\end{proof}

%\begin{remark}
%\textit{Briefly comments on the relation between source condition
%and FBI property.}
%\end{remark}

Assisted with Lemma \ref{lem:scineq}, we are now ready to state an
improved error estimate.
\begin{theorem}
\label{thm:error_est_improved}
Under the conditions in Lemma \ref{lem:scineq} there holds with the
constants stated there
\[
\|x_{\alpha,\beta}^\delta - x^\dagger\|_{\ell^2}\leq
\frac{\delta^2}{2c_1\beta} + \frac{ c_2^2\beta}{2c_1} +
\frac{c_2\delta}{c_1}.
\]
\end{theorem}
\begin{proof} Since $x_{\alpha,\beta}^\delta$ minimizes
$\Phi_{\alpha,\beta}$, the inequality
\begin{equation*}
\frac{1}{2}\|Kx_{\alpha,\beta}^\delta-y^\delta\|^2+\beta\mathcal{R}_\eta(x_{\alpha,\beta}^\delta)\leq
\frac{1}{2}\|Kx^\dagger-y^\delta\|^2+\beta\mathcal{R}_\eta(x^\dagger)
\end{equation*}
holds. Utilizing the fact $\|Kx^\dagger-y^\delta\|\leq\delta$, the
triangle inequality and Lemma \ref{lem:scineq}, we have
\begin{eqnarray*}
\frac{1}{2}\delta^2&\geq&\beta(\mathcal{R}_\eta(x_{\alpha,\beta}^\delta)-\mathcal{R}_\eta(x^\dagger))+
\frac{1}{2}\|Kx_{\alpha,\beta}^\delta-y^\delta\|^2\\ &\geq&\beta
c_1\|x_{\alpha,\beta}^\delta-x^\dagger\|_{\ell^2}-\beta
c_2\|K(x_{\alpha,\beta}^\delta-x^\dagger)\|
+\frac{1}{2}\|Kx_{\alpha,\beta}^\delta-y^\delta\|^2\\ &\geq&\beta
c_1\|x_{\alpha,\beta}^\delta-x^\dagger\|_{\ell^2}-\beta
c_2\|Kx_{\alpha,\beta}^\delta-y^\delta\| -\beta
c_2\delta+\frac{1}{2}\|Kx_{\alpha,\beta}^\delta-y^\delta\|^2.
\end{eqnarray*}
Applying the inequality $ab\leq\frac{1}{2}a^2+\frac{1}{2}b^2$ with
$a= c_2\beta$ and $b=\|Kx_{\alpha,\beta}^\delta-y^\delta\|$
concludes the proof of the theorem.
\end{proof}

\begin{remark}\label{rem:convrate_combined}
We see that for the choice $\beta\sim\mathcal{O}(\delta)$, there
exists some constant $c$ such that
\begin{equation*}
  \|x_{\alpha,\beta}^\delta-x^\dagger\|_{\ell^2}\leq c\delta
\end{equation*}
Hence, the preceding two theorems indicate that the elastic-net
regularization can preserve simultaneously the convergence rate of
classical Tikhonov regularization and that of
$\ell^1$-regularization. The rate of the latter is better, but the
constant $c$ can be huge. Elastic-net regularization remedies this
by retaining the classical rate with a probably more modest
constant. All together, we obtain by combining
Theorems~\ref{thm:error_est_improved} and~\ref{thm:convrate1} that
with $\beta=\delta$ there holds
\begin{equation}
  \label{eq:convrate_combined}
  \|x_{\alpha,\beta}^\delta-x^\dagger\|_{\ell^2}\leq
  \min(c\delta,(1+\|w\|)\sqrt{\delta}).
\end{equation}
\end{remark}

\subsection{A posteriori parameter choice}
\label{subsec:a_posteriori} We now turn to an \textit{a posteriori}
parameter choice rule, i.e.~the discrepancy principle in the sense
of Morozov, for determining the regularization parameter. Note that
\textit{a priori} choice rules usually give only an order of
magnitude instead of a precise value, which undoubtedly impedes
their practical applications. In contrast, the discrepancy principle
enables constructing a concrete scheme for determining an
appropriate regularization parameter. However, there have been
relatively few investigations of \textit{a posteriori} choice rules
for regularization involving general convex functional
\cite{Bonesky2009,Jin2009}. The subsequent developments are
motivated by those in reference \cite{Jin2009}. Mathematically, the
principle amounts to solving a nonlinear equation in $\beta$
\begin{equation}\label{eqn:dp}
   \|Kx_{\alpha,\beta}^\delta-y^\delta\|=\tau\delta
\end{equation}
for some $\tau\geq1$. Without loss of generality, we shall fix
$\tau=1$ in the sequel.

We shall need the next lemma.
\begin{lemma}
The minimizer to the functional $\Phi_{\alpha,\beta}$ vanishes if
and only if  $\alpha\geq\sup_{h\in\ell^2,h\neq0}\frac{\langle K^\ast
y^\delta,h\rangle}{\|h\|_{\ell^1}}$.
\end{lemma}
\begin{proof}
Assume that $0$ is a minimizer of the functional
$\Phi_{\alpha,\beta}$. The minimizing property of $0$ implies that
for any $h\in\ell^2$
\begin{equation*}
  \frac{1}{2}\|y^\delta\|^2\leq\frac{1}{2}\|Kh-y^\delta\|^2+\alpha\|h\|_{\ell^1}+\frac{\beta}{2}\|h\|_{\ell^2}^2
\end{equation*}
Collecting the terms gives
\begin{equation*}
   \langle K^\ast y^\delta,h\rangle\leq \frac{1}{2}\|Kh\|^2+\alpha\|h\|_{\ell^1}+\frac{\beta}{2}\|h\|_{\ell^2}^2.
\end{equation*}
By dividing by $\|h\|_1$ and setting $h=\varepsilon h'$ and letting
$\varepsilon$ tend to zero we deduce that
\begin{equation*}
  \alpha\geq\sup_{h\in\ell^2,h\neq0}\frac{\langle K^\ast y^\delta,h\rangle}{\|h\|_{\ell^1}}.
\end{equation*}

Conversely, assume that the above inequality holds. Then for any
$h\in\ell^2$, there holds
\begin{equation*}
\langle K^\ast y^\delta,h\rangle\leq
\alpha\|h\|_{\ell^1}\leq\frac{1}{2}\|K h\|^2+\alpha\|
h\|_{\ell^1}+\frac{\beta}{2}\|h\|_{\ell^2}^2.
\end{equation*}
By completing square it gives
\begin{equation*}
   \frac{1}{2}\|y^\delta\|^2\leq\frac{1}{2}\|Kh-y^\delta\|^2+\alpha\|h\|_{\ell^1}+\frac{\beta}{2}\|h\|_{\ell^2}^2,
\end{equation*}
By the definition of the minimizer, we conclude that $0$ is the
minimizer of the functional $\Phi_{\alpha,\beta}$.
\end{proof}

The next result shows the existence and uniqueness of the solution
to equation (\ref{eqn:dp}).
\begin{theorem}
Assume that the conditions
$\lim_{\beta\rightarrow0}\|Kx_{\alpha,\beta}^\delta-y^\delta\|<\delta$
and $\|y^\delta\|>\delta$ hold. Then there exists at least one solution
$\beta^\ast$ to equation (\ref{eqn:dp}). Moreover, if the solution
$\beta^\ast$ satisfies
$\beta^\ast\eta<\sup_{h\in\ell^2,h\neq0}\frac{\langle K^\ast
y^\delta,h\rangle}{\|h\|_{\ell^1}}$, then it is also unique.
\end{theorem}
\begin{proof}
Let $\beta_1$ and $\beta_2$ be distinct and for $i=1,2$ denote
$\alpha_i = \eta\beta_i$. By the minimizing property of
$x_{\alpha_1,\beta_1}^\delta$ and $x_{\alpha_2,\beta_2}^\delta$ we
have
\begin{eqnarray*}
&&\frac{1}{2}\|Kx_{\alpha_1,\beta_1}^\delta-y^\delta\|^2+\beta_1\mathcal{R}_\eta(x_{\alpha_1,\beta_1}^\delta)\leq
\frac{1}{2}\|Kx_{\alpha_2,\beta_2}^\delta-y^\delta\|^2+\beta_1\mathcal{R}_\eta(x_{\alpha_2,\beta_2}^\delta),\\
&&\frac{1}{2}\|Kx_{\alpha_2,\beta_2}^\delta-y^\delta\|^2+\beta_2\mathcal{R}_\eta(x_{\alpha_2,\beta_2}^\delta)\leq
\frac{1}{2}\|Kx_{\alpha_1,\beta_1}^\delta-y^\delta\|^2+\beta_2\mathcal{R}_\eta(x_{\alpha_1,\beta_1}^\delta).
\end{eqnarray*}
From these two inequalities we derive
\begin{equation*}
(\|Kx_{\alpha_1,\beta_1}^\delta-y^\delta\|^2-\|Kx_{\alpha_2,\beta_2}^\delta-y^\delta\|^2)(\beta_1-\beta_2)\geq0,
\end{equation*}
i.e.~$\|Kx_{\alpha,\beta}^\delta-y^\delta\|$ is monotonic in
$\beta$. By Corollary \ref{cor:cont}, it is continuous. Therefore
under the conditions
$\lim_{\beta\rightarrow0}\|Kx_{\alpha,\beta}^\delta
-y^\delta\|<\delta$ and $\|y^\delta\|>\delta$, we have
\begin{equation*}
\lim_{\delta\rightarrow0}\|Kx_{\alpha,\beta}^\delta
-y^\delta\|<\delta\quad\mbox{and}\quad
\lim_{\delta\rightarrow\infty}\|Kx_{\alpha,\beta}^\delta
-y^\delta\|=\|y^\delta\|>\delta.
\end{equation*}
The existence of at least one positive solution to equation
(\ref{eqn:dp}) now follows from the continuity.

The optimality condition for $x_{\alpha,\beta}^\delta$ reads
\begin{equation*}
-K^\ast(Kx_{\alpha,\beta}^\delta-y^\delta)-\beta
x_{\alpha,\beta}^\delta\in\beta\eta\partial\|x_{\alpha,\beta}^\delta\|_{\ell^1}
\end{equation*}
Multiplying both sides of the inclusion by $x_{\alpha,\beta}^\delta$
gives
\begin{equation*}
\langle
Kx_{\alpha,\beta}^\delta,Kx_{\alpha,\beta}^\delta-y^\delta\rangle
+\beta\|x_{\alpha,\beta}^\delta\|_{\ell^2}^2+\beta\eta\|x_{\alpha,\beta}^\delta\|_{\ell^1}=0,
\end{equation*}
Under the assumption
$\beta\eta<\sup_{h\in\ell^2,h\neq0}\frac{\langle K^\ast
y^\delta,h\rangle}{\|h\|_{\ell^1}}$, $x_{\alpha,\beta}^\delta$ is
nonzero, and thus for distinct $\beta_1$ and $\beta_2$, the
solutions $x_{\alpha_1,\beta_1}^\delta$ and
$x_{\alpha_2,\beta_2}^\delta$ are distinct.

We show the uniqueness by means of contradiction. Assume that there
exist two distinct solutions $\beta_1$ and $\beta_2$ to equation
(\ref{eqn:dp}). By the minimizing property and distinctness of
$x_{\alpha_1,\beta_1}^\delta$ and $x_{\alpha_2,\beta_2}^\delta$, we
have
\begin{equation*}
\frac{1}{2}\|Kx_{\alpha_1,\beta_1}^\delta-y^\delta\|^2+\beta_1\mathcal{R}_\eta(x_{\alpha_1,\beta_1}^\delta)
<\frac{1}{2}\|Kx_{\alpha_2,\beta_2}^\delta-y^\delta\|^2+\beta_1\mathcal{R}_\eta(x_{\alpha_2,\beta_2}^\delta)
\end{equation*}
which together with
$\|Kx_{\alpha_1,\beta_1}^\delta-y^\delta\|=\|Kx_{\alpha_2,\beta_2}^\delta-y^\delta\|$
implies that
$\mathcal{R}_\eta(x_{\alpha_1,\beta_1}^\delta)<\mathcal{R}_\eta(x_{\alpha_2,\beta_2}^\delta)$.
Reversing the role of $\beta_1$ and $\beta_2$ gives
$\mathcal{R}_\eta(x_{\alpha_2,\beta_2}^\delta)<\mathcal{R}_\eta(x_{\alpha_1,\beta_1}^\delta)$,
which is a contradiction.
\end{proof}

The next result shows the consistency of the discrepancy principle
for elastic-net regularization. We remind that the regularization
parameter $\beta$ determined by the discrepancy principle depends on
both $\delta$ and $y^\delta$, although the dependence is suppressed
for notational simplicity.
\begin{theorem}
Let $\beta$ be determined by equation (\ref{eqn:dp}), and
$x^\dagger$ be the $\mathcal{R}_\eta$-minimizing solution of the
inverse problem. Then we have
\begin{equation*}
    \lim_{\delta\rightarrow0}x_{\alpha,\beta}^\delta=x^\dagger.
\end{equation*}
\end{theorem}
\begin{proof}
By the minimizing property of the solution
$x_{\alpha,\beta}^\delta$, we have
\begin{equation*}
\frac{1}{2}\|Kx_{\alpha,\beta}^\delta-y^\delta\|^2+\beta\mathcal{R}_\eta(x_{\alpha,\beta}^\delta)\leq
\frac{1}{2}\|Kx^\dagger-y^\delta\|+\beta\mathcal{R}_\eta(x^\dagger).
\end{equation*}
This together with equation (\ref{eqn:dp}) and the fact that
$\|Kx^\dagger-y^\delta\|\leq\delta$ indicates that
\begin{equation} \label{eqn:dpineq}
\mathcal{R}_\eta(x_{\alpha,\beta}^\delta)\leq\mathcal{R}_\eta(x^\dagger),
\end{equation}
i.e.~the sequence $\{\mathcal{R}_\eta(x_{\alpha,
\beta}^\delta)\}_\delta$ is uniformly bounded. Therefore the
sequence $\{x_{\alpha,\beta}^\delta\}$ is uniformly bounded, and
there exists a subsequence of $\{x_{\alpha,\beta}^\delta
\}_{\delta}$, also denoted as $\{x_{\alpha,\beta}^\delta\}$, and
some $x^\ast$, such that $x_{\alpha,\beta}^\delta$ converges weakly
to $x^\ast$.

By the triangle inequality, we have
\begin{equation*}
\|Kx_{\alpha,\beta}^\delta-y^\dagger\|\leq\|Kx_{\alpha,\beta}^\delta-y^\delta\|+\|y^\delta-y^\dagger\|\leq
\delta+\delta=2\delta.
\end{equation*}
Therefore, weak lower semicontinuity of the norm gives
\begin{equation*}
0\leq\|Kx^\ast-y^\dagger\|\leq\liminf_{\delta\rightarrow0}\|Kx_{\alpha,\beta}^\delta-y^\delta\|\leq
\lim_{\delta\rightarrow0}2\delta=0,
\end{equation*}
i.e.~$\|Kx^\ast-y^\dagger\|=0$ or $Kx^\ast=y^\dagger$. From
inequality (\ref{eqn:dpineq}), we have
\begin{equation*}
\mathcal{R}_\eta(x^\ast)\leq\liminf_{\delta\rightarrow0}\mathcal{R}_\eta(x_{\alpha,\beta}^\delta)\\ \leq\mathcal{R}_\eta(x^\dagger).
\end{equation*}
Therefore, $x^\ast$ is a $\mathcal{R}_\eta$-minimizing solution, and
by noting the uniqueness of the $\mathcal{R}_\eta$ minimizer
$x^\dagger$, we deduce that $x^\ast=x^\dagger$. Since every
subsequence of $\{x_{\alpha,\beta}^\delta\}$ has a subsequence
weakly converging to $x^\dagger$, the whole sequence weakly
converges to $x^\dagger$.

Furthermore, we have
\begin{eqnarray*}
\mathcal{R}_\eta(x^\dagger)\leq\liminf_{\delta\rightarrow0}\mathcal{R}_\eta(x_{\alpha,\beta}^\delta)
\leq\limsup_{\delta\rightarrow0}\mathcal{R}_\eta(x_{\alpha,\beta}^\delta)\leq\mathcal{R}_\eta(x^\dagger)
\end{eqnarray*}
i.e.
\begin{equation*}
\lim_{\delta\rightarrow0}\mathcal{R}_\eta(x_{\alpha,\beta}^\delta)=\mathcal{R}_\eta(x^\dagger).
\end{equation*}
This together with the weak convergence and Lemma \ref{lem:weakstr}
implies the desired strong convergence.
\end{proof}

The next result shows that the discrepancy principle achieves
similar convergence rates as the \textit{a priori} parameter choice
rule under identical conditions.
\begin{theorem}
Assume that the exact solution $x^\dagger$ satisfies the source
condition (\ref{eq:source_condition_elastic_net}) and the
regularization parameter $\beta$ is determined according to equation
(\ref{eqn:dp}). Then there holds
\begin{equation*}
\|x_{\alpha,\beta}^\delta-x^\dagger\|_{\ell^2}\leq
2\|w\|^{\frac{1}{2}}\delta^{\frac{1}{2}}.
\end{equation*}
Moreover, if the conditions of Lemma~\ref{lem:scineq} hold, then
there holds with the constants given there
\begin{equation*}
\|x_{\alpha,\beta}^\delta-x^\dagger\|_{\ell^2}\leq
\frac{2c_2}{c_1}\delta.
\end{equation*}
\end{theorem}
\begin{proof}
Since $x^\dagger$ satisfies the source condition
(\ref{eq:source_condition_elastic_net}) there exists
$\xi\in\partial\|x^\dagger\|_{\ell^1}$ such that $K^*w = x^\dagger +
\eta\xi$. Inequality (\ref{eqn:dpineq}) implies that
\begin{eqnarray*}
\mathcal{R}_\eta(x_{\alpha,\beta}^\delta)-\mathcal{R}_\eta(x^\dagger)
-\left\langle\eta\xi+x^\dagger,x_{\alpha,\beta}^\delta-x^\dagger\right\rangle
&\leq&-\left\langle\eta\xi+x^\dagger,x_{\alpha,\beta}^\delta-x^\dagger\right\rangle\\
&=&-\left\langle K^\ast
w,x_{\alpha,\beta}^\delta-x^\dagger\right\rangle\\
&=&-\left\langle w,Kx_{\alpha,\beta}^\delta-y^\dagger\right\rangle
\\&\leq&\|w\|\|Kx_{\alpha,\beta}^\delta-y^\dagger\|\leq2\|w\|\delta.
\end{eqnarray*}
However, noting $\xi\in\partial\|x^\dagger\|_{\ell^1}$, we have by
the defining inequality of a subgradient
\begin{eqnarray*}
&&\mathcal{R}_\eta(x_{\alpha,\beta}^\delta)-\mathcal{R}_\eta(x^\dagger)-\left\langle\eta\xi+x^\dagger,x_{\alpha,\beta}^\delta-x^\dagger\right\rangle\\
&=&\frac{1}{2}\|x_{\alpha,\beta}^\delta-x^\dagger\|_{\ell^2}^2
+\eta[\|x_{\alpha,\beta}^\delta\|_{\ell^1}-\|x^\dagger\|_{\ell^1}-\left\langle\xi,x_{\alpha,\beta}^\delta-x^\dagger\right\rangle]\\
&\geq& \frac{1}{2}\|x_{\alpha,\beta}^\delta-x^\dagger\|_{\ell^2}^2
\end{eqnarray*}
Combining above two inequalities gives the first estimate.

By inequality (\ref{eqn:dpineq}) and Lemma \ref{lem:scineq}, we have
\begin{eqnarray*}
c_1\|x_{\alpha,\beta}^\delta-x^\dagger\|_{\ell^2}&\leq&
\mathcal{R}_\eta(x_{\alpha,\beta}^\delta)-\mathcal{R}_\eta(x^\dagger)+c_2\|K(x_{\alpha,\beta}^\delta-x^\dagger)\|\\
&\leq&c_2\|K(x_{\alpha,\beta}^\delta-x^\dagger)\|\leq2c_2\delta
\end{eqnarray*}
This concludes the proof of the theorem.
\end{proof}

\section{Active set algorithms}
\label{sec:algorithms} Having established the analytical properties
of the elastic-net functional and its minimizers, we now proceed to
the algorithmic part of minimizing the functional. We will derive
adaptations of the SSN \cite{griesse2008ssnsparsity} and FSS
\cite{lee2006featuresignsearch} and show that these algorithms are
regularized versions of the respective $\ell^1$-algorithms. In
addition, we show convergence results for both methods. For
notational simplicity, we shall drop the superscript $\delta$ in
this section.

\subsection{Regularized SSN (RSSN)}
\label{subsec:ssn} We now derive an algorithm for the elastic-net
functional $\Phi_{\alpha,\beta}$ based on the SSN
\cite{hintermueller2003primaldualssn,Ulbrich2002}, which in turn
coincides with a regularization of the
SSN~\cite{griesse2008ssnsparsity} and hence will be called RSSN.

Using sub-differential calculus the optimality condition for
$\Phi_{\alpha,\beta}$ reads
\begin{eqnarray*}
0\in\partial\Phi_{\alpha,\beta}(x) = \partial\Psi_{\alpha}(x)+\beta
x. \label{preopt}
\end{eqnarray*}
With the help of the set-valued $\Sign$ function, it reads
\begin{eqnarray}
- K^\ast(Kx-y) -\beta x \in \alpha \Sign(x)\label{opt}.
\end{eqnarray}
The similarity of the optimality conditions for classical $\ell^1$-
and elastic-net minimization suggests adapting existing
$\ell^1$-algorithms. It can be formulated equivalently using the
soft-shrinkage function $S_{\alpha}$, which is defined componentwise by
$S_{\alpha}(x)_i=\max\{0,|x_i|-\alpha\}\cdot \mathrm{sign}(x_i)$.
\begin{lemma}
An element $x$ solves equation~(\ref{opt}) if and only if
\begin{eqnarray}
     F(x):=\beta x-S_\alpha( -K^\ast(Kx-y) )=0.\label{opt1}
\end{eqnarray}
\end{lemma}
\begin{proof}
Obviously, the inclusion (\ref{opt}) is equivalent to
\[
 -\frac{K^\ast(Kx - y)}{\beta} \in x + \frac{\alpha}{\beta}\Sign(x).
\]
Noting the identity $S_\alpha = (\id + \alpha\Sign)^{-1}$ (see,
e.g.~\cite{griesse2008ssnsparsity}) it follows that
\[
x = S_{\alpha/\beta}(-K^\ast(Kx-y)/\beta).
\]
Now the identity $S_{c\alpha}(cx) = cS_\alpha(x)$ for $c>0$
concludes the proof.
\end{proof}

The RSSN consists of solving equation~(\ref{opt1}) by Newton's
method. $F(x)$ is not differentiable in the classical sense because
of the nonsmooth shrinkage operator $S_\alpha$, and thus a
generalized notion of differentiability is required for applying
Newton's method. We shall use the notion of Newton derivative
\cite{chen2000semismooth,hintermueller2003primaldualssn}. The
shrinkage operator $S_\alpha$ turns out to be Newton differentiable.
More precisely, we have the next result.
\begin{lemma}[\cite{griesse2008ssnsparsity}]\label{lemma:newtonrule}
A Newton derivative of $S_\alpha$ is given by
\[
 G(x)=\mat{ \id_{\set{i\in\N}{|x_i|>\alpha}} &0 \\ 0 & 0}
\]
and for any bounded linear operator $T:\ell^2\rightarrow\ell^2$ and any
$b\in\ell^2$ a Newton derivative of $S_\alpha (Tx + b)$ is given by
$G(Tx + b)T$.
\end{lemma}
Hence, a Newton derivative of $F$ is given by $D(x)=\beta\id-G(
-K^\ast(Kx-y) )K^\ast K$. Given a set $A\subset\N$, we split the
operator $K^\ast K$ as
\[
   K^\ast K=\mat{ M_{A} & M_{A A^c} \\ M_{A^cA} & M_{A^c} }.
\]
Upon letting $A_x:=\set{i\in\N}{|K^\ast(Kx-y)|_i>\alpha}$, we have
\[
D(x) = \begin{pmatrix}  \beta \id_{A_x} + M_{A_x} & M_{A_xA_x^c}\\
0    & \beta \id_{A_x^c}\\  \end{pmatrix}.
\]

% \begin{lemma}
% \label{lem:finite_active_set}  Let $x^\ast$ be the minimizer of
% $\Phi_{\alpha,\beta}$ and $x\in\ell^2$. The active set $A_x$ is
% finite, i.e. there exist $\delta>0$ and $i_0\in\N$ such that
% \[
%    \|x-x^\ast\|_{\ell^2}<\delta \implies A_{x}\subset \{1\dots i_0\}.
% \]
% \end{lemma}
% \begin{proof}
% First we observe $K^\ast(Kx^\ast-y) \in \ell^2$ and hence its
% components tend to zero. Therefore there exists some $i_0\in\N$ such
% that $|K^\ast(Kx^\ast-y)|_i < \frac{\alpha}{2}$ for $i>i_0$.
% Moreover, for $\delta<\frac{\alpha}{2\|K^\ast K\|}$ we have
% \begin{eqnarray*}
% |K^\ast(Kx-y)|_i &=& |K^\ast(K(x-x^\ast+x^\ast)-y)|_i\\
% &=& |K^\ast K(x-x^\ast)+K^\ast(Kx^\ast-y)|_i\\
% &\leq& |K^\ast K(x-x^\ast)|_i+|K^\ast(Kx^\ast-y)|_i\\
% &<& \delta \|K^\ast K\| +|K^\ast(Kx^\ast-y)|_i < \alpha,\quad i>i_0.
% \end{eqnarray*}
% This shows the lemma.
% \end{proof}

\begin{lemma}
For every $x\in\ell^2$, $D(x)$ is invertible and $\|D(x)^{-1}\|$ is
uniformly bounded.
\end{lemma}
\begin{proof}
Splitting the equation $D(x)f=g$ blockwise gives
% \[
% \begin{pmatrix} \beta \id_{A_x} + M_{A_x} & M_{A_xA_x^c}\\
%   0    & \beta \id_{A_x^c}\\ \end{pmatrix}
% \mat{f|_{A_x} \\ f|_{A_x^c}}  = \mat{g|_{A_x} \\ g|_{A_x^c}},
% \]
% i.e.
\[
\beta f|_{A_x^c} = g|_{A_x^c}\quad\mbox{and} \qquad (\beta \id_{A_x}
+ M_{A_x}) f|_{A_x} = g|_{A_x} - M_{A_xA_x^c} f|_{A_x^c}.
\]
Therefore, the invertibility of $D(x)$ only depends on the
invertibility of $\beta \id_{A_x} + M_{A_x}$.

Denote by $P_{A_x}$ the canonical projection $P_{A_x}:\ell^2\rightarrow\ell^2$
which projects onto the components listed in ${A_x}$. Then the
matrix $M_{A_x} = P_{A_x} K^\ast KP_{A_x}$, and thus it is
self-adjoint and positive semidefinite.
Therefore, the eigenvalues of $\beta \id_{A_x} + M_{A_x}$ are
contained in the interval $[\beta,\infty)$. Consequently, the matrix
$\beta \id_{A_x} + M_{A_x}$ is invertible, and $\| (\beta \id_{A_x}
+ M_{A_x})^{-1} \| \leq\beta^{-1}$. Now the assertion follows from
\begin{eqnarray*}
\left\| \mat{
\beta \id_{A_x} + M_{A_x} & M_{A_xA_x^c}\\  0    & \beta \id_{A_x^c}\\
}^{-1}g\right\|_{\ell^2} &=&\left\|
\mat{(\beta \id_{A_x} + M_{A_x})^{-1} & -\beta^{-1}(\beta \id_{A_x} + M_{A_x})^{-1} M_{A_xA_x^c}\\
0    & \beta^{-1}\id_{A_x^c}}g
\right\|_{\ell^2}\\
&\leq&\beta^{-1}(\|g_{A_x}\| + \beta^{-1}\|
M_{A_xA_x^c}g_{A_x^c}\| + \|g_{A_x^c}\|)\\
&\leq& \beta^{-1}( 2 + \beta^{-1}\|K^\ast K\|) \cdot \|g\|,
\end{eqnarray*}
where we have used the inequality $\| M_{A_xA_x^c}\|\leq\|K^\ast
K\|$.
\end{proof}

This lemma verifies the computability of Newton iterations to find
solutions of \eqref{opt1}:
\begin{eqnarray*}
        x^{k+1} &=& x^k - D(x^k)^{-1} F(x^k)\\
%&=& x_k - \mat{(\beta \id_{A_{x_k}} + M_{A_{x_k}})^{-1}
%  & -\beta^{-1}(\beta \id_{A_{x_k}} + M_{A_{x_k}})^{-1} M_{A_{x_k}A_{x_k}^c}\\
%  0    & \beta^{-1}\id_{A_{x_k}^c}}
%  \mat{(\beta {x_k}+K^\ast(K{x_k}-y) \pm \alpha)|_{A_{x_k}} \\
%  \beta {x_k}|_{A_{x_k}^c}}\\
%  &=& x_k - \mat{(\beta \id_{A_{x_k}} + M_{A_{x_k}})^{-1}
%  (\beta x_{k}|_{A_{x_k}}+(K^\ast Kx_k)|_{A_{x_k}}-(K^\ast y)|_{A_{x_k}}\mp\alpha
%  - M_{A_{x_k}A_{x_k}^c}x_{k}|_{A_{x_k}^c})\\
%  x_{k}|_{A_{x_k}^c}}\\
%  &=& x_k - \mat{(\beta \id_{A_{x_k}} + M_{A_{x_k}})^{-1}
%  (\beta x_{k}|_{A_{x_k}}+M_{A_{x_k}} x_k|_{A_{x_k}}-(K^\ast y)|_{A_{x_k}}\mp\alpha)\\
%  x_{k}|_{A_{x_k}^c}}\\
  &=& \mat{
  (\beta \id_{A_{x^k}} + M_{A_{x^k}})^{-1}  (  (K^\ast y)|_{A_{x^k}} \pm \alpha)\\
  0}.
\end{eqnarray*}
In particular, this shows that the next iterate depends on the
previous one only via the active set.

We are now ready to state the complete algorithm.
\begin{itemize}
\item[\textbf{Step 1}] Initialize: $k=0$, $x^0=0$
\item[\textbf{Step 2}] Choose active set $A_{x^k} =
\set{i\in\N}{|K^\ast(Kx^k-y)|_i>\alpha}$ and calculate
\[
s^{k}_i = \begin{cases}
1, & [-K^\ast(Kx^k-y)]_i > \alpha\\
-1, & [-K^\ast(Kx^k-y)]_i < -\alpha\\
0, & \text{else}\end{cases}.
\]
\item[\textbf{Step 3}] Update for the next iterate $x^{k+1}$
\begin{equation}\label{eq:update_ssn}
\begin{split}
&x^{k+1}|_{A_{x^k}} = (\beta \id_{A_{x^k}}+M_{A_{x^k}})^{-1}( K^\ast
y-s^k \alpha )|_{A_{x^k}}\\
&x^{k+1}|_{A_{x^k}^c} = 0
\end{split}
\end{equation}
\item[\textbf{Step 4}] Check stopping criteria. Return $x^{k+1}$ as a solution or set $k\leftarrow k+1$
and repeat from step 2.
\end{itemize}

A natural stopping criterion for the algorithm is the change of the
active set. If it does not change for two subsequent iterations,
then a minimizer has been attained. The next result is well known
and included for completeness.
\begin{theorem}
Let $K:\ell^2 \rightarrow \mathcal{H}_2$ and $\alpha,\beta>0$. The
RSSN converges locally superlinearly.
\end{theorem}
\begin{proof}
Let $x^\ast$ be the minimizer of $\Phi_{\alpha,\beta}$. Using the
above lemmas and $F(x^\ast)=0$ we have
\begin{eqnarray*}
\|x^{k+1}-x^\ast\|_{\ell^2} &=& \|x^k-D(x^k)^{-1}F(x^k)-x^\ast\|_{\ell^2}\\
&=& \|x^k-D(x^k)^{-1}F(x^k)-x^\ast +D(x^k)^{-1}F(x^\ast)\|_{\ell^2}\\
&=& \|D(x^k)^{-1}\| \|D(x^k)(x^k-x^\ast)-F(x^k)+F(x^\ast)\|.
\end{eqnarray*}
The definition of Newton derivative implies
\[
\lim_{x\rightarrow
x^\ast}\frac{\|F(x)-F(x^\ast)-D(x)(x-x^\ast)\|}{\|x-x^\ast\|_{\ell^2}}=0,
\]
and hence for arbitrary $\eps>0$ and $\|x^k-x^\ast\|_{\ell^2}$ sufficiently
small we have
\begin{eqnarray*}
\|D(x^k)^{-1}\| \|D(x^k)(x^k-x^\ast)-F(x^k)+F(x^\ast)\| <
\|D(x^k)^{-1}\|\cdot \eps \|x^k-x^\ast\|_{\ell^2}
\end{eqnarray*}
which shows the desired superlinear local convergence.
\end{proof}

\begin{remark}
Several comments on the algorithm are in order. Firstly, this
algorithm differs from the classical SSN
\cite{griesse2008ssnsparsity} only in the regularization of the
equation in step 3. Secondly, the proposed RSSN method is different
from the standard regularized Newton method (also known as the
Levenberg-Marquardt method) via
\[
x^{k+1}=x^k- (D(x^k) + \eta\id)^{-1} F(x^k),
\]
for some $\eta>0$, in that the latter regularizes globally whereas
the former regularizes only on the active set. Thirdly, there are
several equivalent reformulations of the minimization problem. For
instance, multiplying \eqref{opt} by $\gamma>0$ and adding $x$ gives
\begin{eqnarray}
x - \gamma K^\ast(Kx - y) - \gamma\beta x \in x +
\gamma\alpha\Sign(x),\label{choice2}
\end{eqnarray}
and also an alternative characterization of a minimizer of
$\Psi_{\alpha,\beta}$: $x - S_{\gamma\alpha}(x-\gamma K^\ast(Kx - y)
- \gamma\beta x) = 0$. It leads to a similar algorithm but with a
different active set, i.e.
\begin{equation*}
A^1_{x}=\set{i\in\N}{|x- \gamma K^\ast(Kx - y) - \gamma\beta
x|_i>\gamma\alpha}
\end{equation*}
Another choice of the active set derives by rewriting
\eqref{choice2} as $x - \gamma K^\ast(Kx - y) \in (1+\gamma\beta)x +
\gamma\alpha\Sign(x)$. This gives $(1+\gamma\beta)x -
S_{\gamma\alpha}(x - \gamma K^\ast(Kx - y) ) = 0$, and also a third
choice of the active set
\begin{align*}
A^2_{x}&=\set{i\in\N}{|x - \gamma K^\ast(Kx - y)|_i>\gamma\alpha}.
\end{align*}
These different choices may affect the convergence behavior of the
respective algorithms.
\end{remark}

\subsection{Regularized FSS (RFSS)}\label{subsec:fss}
The main drawback of the RSSN is its potential lack of global
convergence. Globalization may be achieved by adopting alternative
selection strategies for the active set. The RFSS is one such
example. It derives from the FSS \cite{lee2006featuresignsearch} as
the RSSN from the SSN. In this section we will describe the RFSS
algorithm in detail and show the next convergence result. For
simplicity, we consider only finite-dimensional problems:
$K:\R^s\rightarrow\R^m$, $y\in\R^m$ and
$\underline{s}=\{1,2,\ldots,s\}$.

\begin{theorem}\label{thm:fss}
The RFSS converges globally in finitely many steps, moreover every
iteration strictly decreases the value of the functional
$\Phi_{\alpha,\beta}$.
\end{theorem}

We shall need the notion of consistency, which plays a fundamental
role in the RFSS.
\begin{definition}
Let $A\subset\underline{s}$, $x=(x_i)_{i\in\underline{s}}\in\R^s$ and
$\theta=(\theta_i)_{i\in\underline{s}}\in\{-1,0,1\}^s$. The triple
$(A,x,\theta)$ is called consistent if
\begin{eqnarray*}
i\in A &\implies& \sign(x_i)=\theta_i\neq0,\\
i\in A^c &\implies& x_i=\theta_i = 0.
\end{eqnarray*}
\end{definition}

With a consistent triple $(A,x,\theta)$ we can split the optimality
condition \eqref{opt} into
\begin{eqnarray}
&(-K^\ast(Kx-y) -\beta x)_i = \alpha \theta_i, \quad &i\in
A,\label{conda}\\
&|K^\ast(Kx-y)|_i \leq \alpha, \quad &i\in A^c.\label{condb}
\end{eqnarray}

\begin{remark}\label{rem:consistency}
The formulas \eqref{conda} and \eqref{condb} correspond to the
optimality condition for the following auxiliary functional
\begin{equation}
\Phi_{\alpha,\beta,\theta}(x)=\frac{1}{2}\|Kx-y\|^2+\alpha
\langle x,\theta \rangle + \frac{\beta}{2}\|x\|_{\ell^2}^2.
\label{phiabt}
\end{equation}
By the definition of consistency, $\Phi_{\alpha,\beta}(x) =
\Phi_{\alpha,\beta,\theta} (x)$ if $\sign(x)_i=\theta_i$
for all nonzero components of $x$. In any case we have
\[
\Phi_{\alpha,\beta,\theta}(x) \leq \Phi_{\alpha,\beta}(x).
\]
\end{remark}

Now we are ready to describe the complete RFSS algorithm in five
steps. The description will also provide a constructive proof of
Theorem \ref{thm:fss}.

\begin{itemize}
\item[\textbf{Step 1}] Initialize: $k=1$, $A_0=\emptyset$, $x^0=0$ and
$\theta^0=0$. Any consistent triple $(A_0,x^0,\theta^0)$ is valid
for initialization. Then check the optimality condition (\ref{opt})
and take one of the actions
\begin{itemize}
\item[(i)] return the solution if fulfilled;
\item[(ii)] continue with Step 2 if \eqref{condb} is not fulfilled;
\item[(iii)] continue with Step 3 otherwise.
\end{itemize}

\item[\textbf{Step 2}]
At this step, the following premises hold: The optimality condition
\eqref{condb} is not fulfilled and the triple
$(A_{k-1},x^{k-1},\theta^{k-1})$ is consistent. This step performs a
greedy scheme by selecting the index $i^k_0$ violating condition
\eqref{condb} the most, i.e.
\[
i^k_0 \in \argmax_{i\in A_{k-1}^c}  |K^\ast(Kx^{k-1}-y)|_i -\alpha.
\]
Then update the active set by $A_{k} = A_{k-1}\cup\{i^{k}_0\}$,
update $\theta^k$ by
\[
\theta^k_i= \begin{cases}
\theta^{k-1}_i, &i\neq i^k_0\\
-\sign( (K^*(Kx^{k-1}-y))_{i^k_0} ), & i=i^k_0\\
\end{cases}
\]
and continue with Step 3.
\item[\textbf{Step 3}] Calculate the next iterate $x^k$ such that
\eqref{conda} is fulfilled, i.e.~$x^k$ is optimal for
$\Phi_{\alpha,\beta,\theta^{k}}$, by
\begin{equation*}
x^{k}|_{A_k} = (\beta\id+M_{A_k})^{-1}( K^\ast y - \alpha\theta^{k}
)|_{A_k}\quad \mbox{and}\quad x^{k}|_{A^c_k} = 0.
\end{equation*}
Observe that the update coincides with that in the RSSN. If the
triple $(A_k,x^k,\theta^k)$ is consistent, continue with Step 5, and
otherwise  continue with Step 4. For the former, we deduce from
Remark~\ref{rem:consistency} that
\[
\Phi_{\alpha,\beta}(x^k)=\Phi_{\alpha,\beta,\theta^{k}}(x^k)<\Phi_{\alpha,\beta,\theta^{k}}(x^{k-1})
\leq \Phi_{\alpha,\beta}(x^{k-1}).
\]
\item[\textbf{Step 4}]
This step handles inconsistent $(A_k,x^k,\theta^k)$. We consider two
different cases separately.
\begin{itemize}
\item[Case 1] The preceding step of Step 3 is Step 4, i.e.
    $(A_k,x^{k-1},\theta^k)$ is consistent. Therefore, there must be at
    least one index such that the signs of $x^k$ and $x^{k-1}$ differ.
    Let $\lambda_0$ the smallest $\lambda\in(0,1)$ such that
    \[
    \exists i_0\in A_k:\quad (\lambda x^k+(1-\lambda) x^{k-1})_{i_0}=0,
    \]
    and denote $x_{\lambda_0} = \lambda_0 x^k+(1-\lambda_0) x^{k-1}$.
    Now the convexity of $\Phi_{\alpha,\beta,\theta^{k}}$ implies
    \begin{eqnarray*}
    \Phi_{\alpha,\beta}(x_{\lambda_0}) &=& \Phi_{\alpha,\beta,\theta^{k}}(x_{\lambda_0}) \\
    &\leq& \lambda_0 \Phi_{\alpha,\beta,\theta^{k}}(x^{k})+(1-\lambda_0) \Phi_{\alpha,\beta,\theta^{k}}(x^{k-1}) \\
    &<& \lambda_0 \Phi_{\alpha,\beta,\theta^{k}}(x^{k-1})+(1-\lambda_0) \Phi_{\alpha,\beta,\theta^{k}}(x^{k-1}) \\
    &=&\Phi_{\alpha,\beta,\theta^{k}}(x^{k-1})=\Phi_{\alpha,\beta}(x^{k-1}),
    \end{eqnarray*}
    by the minimizing property of $x^k$ for
    $\Phi_{\alpha,\beta,\theta^k}$. Now we update $(A_k,x^k,\theta^k)$
    by
    \begin{eqnarray*}
    x^k\leftarrow x_{\lambda_0},\quad A_k \leftarrow
    \set{i\in\underline{s}}{x^k_i\neq0},\quad \theta^k \leftarrow \sign(x^k)
    \end{eqnarray*}
    and check equation \eqref{conda}. If fulfilled continue with step 5,
    otherwise increase $k$ by one and continue with step 3.
\item[Case 2] The preceding step of Step 3 is Step 2, i.e.
    $|K^\ast(Kx^{k-1}-y)|_{i_0^k}>\alpha$ and $x^{k-1}_{i_0^k}=0$. The choice of
    $\theta^k_{i_0^k}$ implies
    \[
    \sign(\nabla\Phi_{\alpha,\beta,\theta^k}(x^{k-1}))_{i_0^k}=\sign((K^\ast(Kx^{k-1}-y))_{i_0^k}
    + \alpha\theta^k_{i_0^k})= -\theta^k_{i_0^k},
    \]
    and moreover,
    \[
    \nabla\Phi_{\alpha,\beta,\theta^k}(x^{k-1})|_{A_{k-1}} = 0.
    \]
    Now the Taylor expansion of $\Phi_{\alpha,\beta,\theta^k}$ at
    $x^{k-1}$ yields that for $\tilde{x}$ near to $x^{k-1}$ with $(\tilde{x})_{A_{k-1}^c\backslash \{i_0^k\}}
    =(x^{k-1})_{A_{k-1}^c\backslash\{i_0^k\}}=0$
    %\[
    %0>(\Phi_{\alpha,\beta,\theta^k}(\tilde{x})-\Phi_{\alpha,\beta,\theta^k}(x^{k-1}))_{i_0^k}
    %=(\nabla\Phi_{\alpha,\beta,\theta^k}(x^{k-1})(\tilde{x}-x^{k-1}))_{i_0^k}
    %= (\nabla \Phi_{\alpha,\beta,\theta^k}(x^{k-1}) \tilde{x})_{i_0^k},
    %\]
    \[
    0>\Phi_{\alpha,\beta,\theta^k}(\tilde{x})-\Phi_{\alpha,\beta,\theta^k}(x^{k-1})
    =(\nabla\Phi_{\alpha,\beta,\theta^k}(x^{k-1}))_{i_0^k}(\tilde{x}-x^{k-1})_{i_0^k}
    = (\nabla \Phi_{\alpha,\beta,\theta^k}(x^{k-1}))_{i_0^k}(\tilde{x})_{i_0^k},
    \]
    by observing $(x^{k-1})_{i_0^k}=0$, which consequently implies
    \[
    0>-\theta^k_{i_0^k}\tilde{x}_{i_0^k} \implies
    \theta^k_{i_0^k}=\sign{\tilde{x}_{i_0^k}}.
    \]
    The minimizing property in Step 3, implies that
    $\Phi_{\alpha,\beta,\theta^k}(x^{k})<\Phi_{\alpha,\beta,\theta^k}(x^{k-1})$,
    which further implies together with the convexity of
    $\Phi_{\alpha,\beta,\theta^k}$ that there exists a $\tilde{x}$ near
    to $x^{k-1}$ on the line segment from $x^{k-1}$ to $x^k$ such that
    $\Phi_{\alpha,\beta,\theta^k}(\tilde{x})<\Phi_{\alpha,\beta,\theta^k}(x^{k-1})$.
    Consequently,
    \[
    \sign(x^k_{i_0^k}) = \theta^k_{i_0^k}.
    \]
    Thus there can be a sign change for $x^{k-1}$ to $x^k$ only on a
    component other than $i_0^k$. Now we continue analogously to Case 1.
\end{itemize}
\item[\textbf{Step 5}] At this step, the following premises are fulfill: $(A_k,x^k,\theta^k)$ is
consistent and the optimality condition \eqref{conda} is fulfilled.
Check \eqref{condb}. If fulfilled, stop, otherwise continue with
Step 2.
\end{itemize}

From the strictly reducing property of the algorithm we know that
every possible active set is attained at most once. This guarantees
the convergence in finitely many steps. We observe that the
reduction properties hold also for infinite-dimensional problems.

\begin{remark}
        Both algorithms, RSSN as well as RFSS, also work for weighted $\ell^1$ norms, i.e. replacing
        $\alpha \|x\|_{\ell^1}$ by $\sum \alpha_i |x_i|$ in the definition of $\Phi_{\alpha,\beta}$
        where $\alpha_i\geq c>0$ for some constant $c$.
        All theoretical results remain valid for this case.
\end{remark}
\begin{remark}
  The symmetric matrix in Step 3 of the RFSS changes only by one row and one
  column at almost every iteration. Therefore, it is advisable to use
  the Cholesky factorization to solve the equation because of its
  straightforward update and reduction in computational efforts.
\end{remark}

\section{Numerical experiments}\label{sec:experiments}

In this section we compare classical $\ell^1$-minimization algorithms 
and their elastic-net counterparts for both well- and ill-conditioned 
operator equations. For qualitative properties of elastic-net
regularization compared to classical $\ell^1$-regularization, we
refer to \cite{Zou.etal:2008}, and for in-depth comparisons of
existing $\ell^1$ algorithms, we refer to
\cite{griesse2008ssnsparsity,Loris2009}. We only aim at illustrating 
algorithmic differences between SSN and RSSN or FSS and RFSS. All the
algorithms were implemented in MATLAB R2008a and run on an AMD Athlon 64 X2 
Dual Core Processor 3800+ equipped with a 64 bit linux.
% MATLAB programs of the algorithms are available at \texttt{http://??}

\subsection{Test 1: Well conditioned operators and absence of noise}

As for our first test, we use a setting that $K$ is a $400\times 400$ 
Gaussian random matrix with its columns normalized to unit norm. This 
gives rise to a well-conditioned operator and hence should pose no problem 
to classical SSN and FSS. The exact solution $x^\dagger$ is the zero
vector with every 10th entry set to 1. The simulated exact data $y^\dagger$
is then generated by $y^\dagger=Kx^\dagger$.

To study the influence of the parameter $\beta$, we fix the value of 
the parameter $\alpha$ at $\alpha=10^{-5}$. The numerical results 
for one typical realization of the random matrix are summarized in 
Table \ref{tab:test1}. In the table, $x^\ast$ denotes the minimizer
computed by the algorithm at hand, $e_{x^*}:=\|x^\dagger-x^*\|_{
\ell^2}/\|x^\dagger\|_{\ell^2}$ denotes the relative error, 
$\#A_{x^\ast}$ refers to the size of the active set $A_{x^\ast}$,
indicating the sparsity of the solution $x^\ast$, and the computing time is measured in milliseconds (ms). Note that
$\beta=0$ corresponds to the classical $\ell^1$ algorithms.
 
The parameter $\beta$ affects significantly the sparsity 
of the minimizer, especially in case of larger values, e.g. $\beta=2^{-12}$. 
This value renders the dominance of the $\ell^2$ term over the 
$\ell^1$ term in the functional, and thus completely destroys the desired sparsity. 
Meanwhile, it also deteriorates greatly the reconstruction accuracy 
and computational efficiency. The latter is due to the fact that 
more iterations are needed to accurately resolve all the entries 
in the active set. This is the case for both RFSS and RSSN. 
For small values of $\beta$, the computing time changes only 
slightly.

\begin{table}
\centering \caption{Numerical results for Test 1: a well conditioned problem with exact data.
% influence of $\beta$ on sparsity of the minimizer $x^\ast$, number of iterations and relative error.
}
\label{tab:test1}
\begin{tabular}{ccccccc}
\toprule
\multicolumn{3}{}{} & \multicolumn{2}{c}{RFSS} & \multicolumn{2}{c}{RSSN}\\
\cmidrule(r){4-5}\cmidrule(l){6-7}
 $\beta$   & \#$A_{x^*}$  & $e_{x^*}$      & \#iterations & time(ms) & \#iterations & time(ms)\\
\midrule  
	0        & 70           & 6.61e-7        & 80           & 83       & 7  & 82\\
  $2^{-30}$&  70          &6.64e-7         & 80           & 85       & 7  & 83\\
  $2^{-28}$&  70          &6.75e-7         & 80           & 91       & 8  & 87\\
  $2^{-24}$&  95          &9.44e-7         & 105          & 107      & 7  & 84\\
  $2^{-20}$& 181          &9.68e-6         & 193          & 308      & 8  & 98\\
  $2^{-16}$& 186          &1.68e-4         & 200          & 330      & 7  & 94\\
  $2^{-12}$& 388          &8.59e-2         & 556          & 3490     & 15 & 287\\
\bottomrule
\end{tabular}
\end{table}

\subsection{Test 2: Rank-deficient operators and absence of noise}
The next test demonstrates the stability of elastic-net
in the more challenging case of ill-conditioned or  rank-deficient operators. 
We use the same setting as for Test 1, but set columns 201 to 400
of the random matrix $K$ the same as columns 1 to 200. This gives
rise to a rank-deficient matrix. The numerical results for one exemplary
random matrix are shown in Table \ref{tab:test2}.

% still the variance with respect to different random matrices is low. 

As expected, both SSN and FSS fail ruthlessly as a consequence of inverting 
rank-deficient submatrices. In sharp contrast to these classical $\ell^1$ 
algorithms, their elastic-net counterparts remain robust so long as the 
$\beta$ value is not exceedingly small. These algorithms converge and give 
results with accuracy comparable to the well-conditioned case, see Tables 
\ref{tab:test1} and \ref{tab:test2}.

\begin{table}
\centering \caption{Numerical results for Test 2: a rank-deficient problem with exact data. } % The methods for classical $\ell^1$
% minimization fail, while elastic-net still gives notable results.}
\label{tab:test2}
\begin{tabular}{ccccccc}
\toprule
\multicolumn{3}{}{} & \multicolumn{2}{c}{RFSS} & \multicolumn{2}{c}{RSSN}\\
\cmidrule(r){4-5}\cmidrule(l){6-7}
 $\beta$   & \#$A_{x^*}$  & $e_{x^*}$      & \#iterations & time(ms) & \#iterations & time(ms)\\
\midrule   
	0       & -            & -              & -            & -        & -            & -\\
  $2^{-24}$&  70          & 4.65e-7        & 96           & 98       & -            & -\\
  $2^{-20}$& 218          & 2.61e-6        & 218          & 461      & 5            & 96\\
  $2^{-16}$& 218          & 4.23e-5        & 220          & 449      & 5            & 93\\
  $2^{-12}$& 360          & 1.03e-3        & 368          & 1670     & 6            & 142\\
 \bottomrule
\end{tabular}
\end{table}

For both Tests 1 and 2, we observe that RSSN typically takes fewer 
iterations than RFSS, but it works on bigger active sets during the 
iteration. Despite this apparent difference, the computing time for 
both algorithms is practically identical on these datasets in case 
of small $\beta$ values. For larger $\beta$s, the support of the
minimizer gets larger and hence, RFSS needs considerably more
iterations and consequently more computing time, whereas for RSSN 
the number of iterations stays small because the active set can
change more dramatically.

Finally, we would like to note that for both tests, the variation of 
the numerical results with respect to different realizations of the 
random matrix $K$ as well as its size is fairly small, and thus the 
algorithms are statistically robust. We omitted the results for these
variants as they are similar to those presented herein.

\subsection{Test 3: Presence of noise}
Next we investigate the practically more relevant case of noisy data.
To this end, we use again the settings of Tests 1 and 2 but adding $5\%$
Gaussian noise to the exact data $y^\dagger$ to get $y^\delta$.
The value of the regularization parameter $\alpha$ is set to
$\alpha \sim \delta = \|y^\delta-y^\dagger\|$.
In addition, we also keep track of the error $e_{Kx^*}:=\|y^\dagger-
Kx^*\|/\|y^\dagger\|$. To assess the 
statistical performance of the algorithms, we repeat the experiment 
100 times and calculate the mean values. The numerical results for
the well-conditioned operator are shown in Table \ref{tab:noisy}.
Analogue results can be obtained for 
the rank-deficient operator, see Table \ref{tab:noisy2}.

\begin{table}
\centering \caption{Numerical results for Test 3: a well conditioned 
problem with noisy data.}\label{tab:noisy}
\begin{tabular}{cccccccc}
\toprule
\multicolumn{4}{}{} & \multicolumn{2}{c}{RFSS} & \multicolumn{2}{c}{RSSN}\\
\cmidrule(r){5-6}\cmidrule(l){7-8}
 $\beta$   & \#$A_{x^*}$  & $e_{x^*}$  & $e_{Kx^*}$     & \#iterations & time(ms) & \#iterations & time(ms)\\
\midrule   
0 & 67.74 & 3.02 & 0.29 & 67.74 & 71.01 & 8.29 & 75.03 \\
$2^{-8}$ & 68.08 & 2.98 & 0.28 & 68.08 & 76.59 & 7.36 & 75.88 \\
$2^{-5}$ & 70.43 & 2.72 & 0.26 & 70.43 & 79.25 & 6.59 & 68.22 \\
$2^{-3}$ & 77.01 & 2.16 & 0.21 & 77.01 & 80.30 & 5.92 & 62.05 \\
$2^{-1}$ & 96.28 & 1.44 & 0.15 & 96.28 & 104.76 & 5.08 & 62.83 \\
       \bottomrule
\end{tabular}
\end{table}

\begin{table}
\centering
\caption{Numerical results for Test 3: a 
rank-deficient problem with noisy data.}\label{tab:noisy2}
\begin{tabular}{cccccccc}
\toprule
\multicolumn{4}{}{} & \multicolumn{2}{c}{RFSS} & \multicolumn{2}{c}{RSSN}\\
\cmidrule(r){5-6}\cmidrule(l){7-8}
 $\beta$   & \#$A_{x^*}$  & $e_{x^*}$  & $e_{Kx^*}$     & \#iterations & time(ms) & \#iterations & time(ms)\\
\midrule   
0 & - & - & - & - & - & - & - \\
$2^{-8}$ & 86.14 & 1.80 & 0.24 & 86.14 & 87.76 & 5.93 & 61.81 \\
$2^{-5}$ & 87.30 & 1.73 & 0.23 & 87.30 & 89.67 & 5.71 & 61.19 \\
$2^{-3}$ & 91.66 & 1.55 & 0.21 & 91.66 & 97.57 & 5.40 & 61.71 \\
$2^{-1}$ & 106.40 & 1.24 & 0.17 & 106.40 & 121.42 & 4.86 & 60.81 \\
       \bottomrule
\end{tabular}
\end{table}

We observe similar performance for the algorithms in terms of 
the number of iterations and computation time as the noise-free case. 
But the parameter $\beta$ now plays a far less influential role. This 
is attributed to the larger residual $\|Kx^\ast-y^\delta\|$. The 
presence of noise in the data inevitably deteriorates the accuracy 
of the results, compare Table \ref{tab:noisy} with Table \ref{tab:test1}. 
In the well-conditioned case, the elastic-net algorithms give results 
with comparable sparsity within commensurate
computing time as that of classical $\ell^1$ algorithms. However,
incorporating the $\ell^2$ term into the functional can 
improve the accuracy of the result, and more important, restore
the stability of the algorithms, which is especially evident in
case of rank-deficient operators, see Table \ref{tab:noisy2}.

\subsection{Test 4: Convergence rates}
The next experiment studies the convergence rates with respect to the noise 
level $\delta$, see Theorems~\ref{thm:convrate1} and \ref{thm:error_est_improved}. 
Here we utilize the blur problem from MATLAB regularization tools~\cite{Hansen2007} 
with the following parameters: image size $50\times50$, band $5$, sigma $0.7$. 
We calculate the minimizer $x^\ast$ of $\Phi_{\alpha,\beta}$ using RSSN with $
\alpha=\delta$ and each $\beta\in\{0, \alpha/4,\alpha/2,\alpha\}$.
Figure \ref{fig:konvRate} displays the noise levels $\delta$ ($x$-axis) 
and the respective errors $\|x_{\alpha,\beta}^\delta-x^*\|_{\ell^2}$ 
($y$-axis) in a doubly logarithmic scale. In the figure,
the line from bottom to top corresponds respectively to the 
results for $\beta=0, \alpha/4,
\alpha/2$ and $\alpha$.

\begin{figure}
\begin{center}
        \begin{tikzpicture}[xscale=.3,yscale=.3]
                \begin{scope}[>=latex]%style für Pfeile
                        %koordinatenkreuz mit beschriftungen
                  \draw[->] (-10,-6) -- (2.8,-6) node[right] {$\log\delta$};
                  \draw[->] (-10,-6) -- (-10,6.8) node[above] {$\log\|x^\dagger-x^\ast\|_{\ell^2}$};
                \end{scope}
                %Achseneinteilungen
    \foreach \x in {-10,-8,...,2}
      \draw (\x,-5.9) -- (\x,-6) node[below] {\footnotesize $\x$};
    \foreach \y in {-6,-4,...,6}
      \draw (-10,\y) -- (-9.9,\y) node[left] {\footnotesize $\y$};

    %daten aus datei
    \draw plot  file {./fig1Eta1};
    \draw plot  file {./fig1Eta.5};
    \draw plot  file {./fig1Eta.25};
    \draw plot  file {./fig1Eta0};
  \end{tikzpicture}
\caption{Log-log plot of the reconstruction error versus the noise
level $\delta$. The lines correspond from bottom to top to the error
using $\alpha=\delta$ and $\beta= 0$, $\alpha/4$,
$\alpha/2$, $\alpha$.}\label{fig:konvRate}
\end{center}
\end{figure}
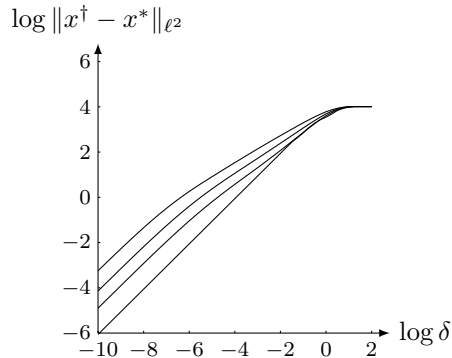

The results shown in Figure~\ref{fig:konvRate} corroborate the
estimate~(\ref{eq:convrate_combined}) in 
Remark~\ref{rem:convrate_combined}: For large $\beta$ values and high 
noise levels we observe $\|x_{\alpha,\beta}^\delta - x^*\|_{\ell^2} 
\approx C \delta^{0.61}$, which is in agreement with the square-root-like 
estimate, while in case of lower noise levels we observe 
$\|x_{\alpha,\beta}^\delta - x^*\|_{\ell^2} \approx C 
\delta^{0.99}$, i.e. the slope is close to unit, which corresponds 
to the improved convergence rate of $\mathcal{O}(\delta)$.

Increasing the value of the standard deviation, i.e. sigma, of the 
blur function makes the problem more ill-posed. The numerical results 
for sigma $=10$ are shown in Table \ref{tab:reg}. One can clearly see 
the regularizing effect of elastic-net compared to classical
$\ell^1$ minimization: The algorithms for the latter do not converge 
for low noise levels, i.e. small $\alpha$ values and $\beta=0$.  
Upon decreasing noise level one observe the growing influence of the 
ill-conditioning of the operator, which consequently leads to 
numerical troubles for the classical $\ell^1$ algorithms.

\begin{table}
\centering \caption{Numerical results for Test 4: blur problem with 
various levels of noise in the data. 
% regularizing property of $\beta$. For small noise levels $\delta$ the increasing 
%influence of the ill-conditioned operator leads to numerical
% troubles for the unregularized algorithm i.e. $\beta=0$.
}\label{tab:reg}
\begin{tabular}{ccccc}
\toprule
\multicolumn{1}{}{} & \multicolumn{4}{c}{$e_{x^\ast}$}\\
\cmidrule(l){2-5}
$\log(\delta)$ & $\beta=\alpha$ & $\beta=\alpha/2$ & $\beta=\alpha/4$ & $\beta=0$\\
\midrule
-9 & 0.31 & 0.25 & 0.20 & -\\
-7 & 0.46 & 0.41 & 0.35 & -\\
-5 & 0.73 & 0.67 & 0.63 & 1.13\\
-3 & 1.00 & 1.00 & 1.00 & 1.00\\
\bottomrule
\end{tabular}
\end{table}

Finally we add $5\%$ noise into the blurred image. The reconstructed 
images for $\alpha=4\times10^{-4}$ and different $\beta$ values are shown
in Figure~\ref{fig:blur}. For this example none of the tested $\ell^1$ 
algorithms would converge. However, a path-following strategy can remedy
the problem: decreasing the $\beta$ value gradually, and using the elastic-net 
reconstruction for a larger $\beta$ value as the initial guess for the RSSN 
iterations with a smaller $\beta$ value. This is in accordance with Proposition 
\ref{lem:beta_to_zero}. Numerically, by iterating this procedure we can 
then obtain an acceptable $\ell^1$-reconstruction. The reconstructions show also clearly
the qualitative differences between elastic-net and $\ell^1$-minimization: For 
the former, neighboring pixels tend to feature groupwise structure, whereas 
for the latter, neighboring pixels more or less behave independent of each 
other.

\begin{figure}
\begin{center}
\includegraphics[width=\textwidth]{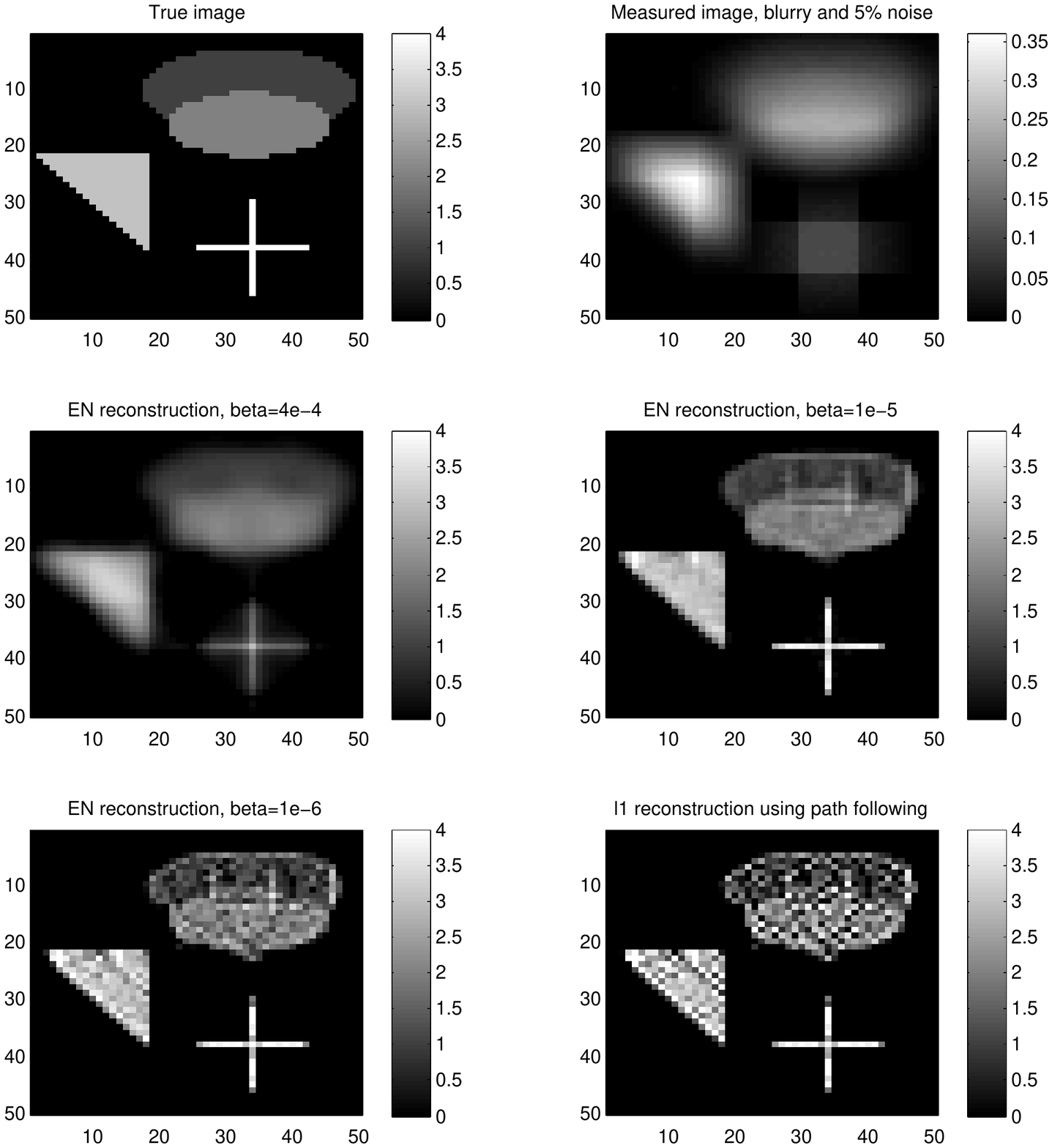}
\end{center}
\caption{A true image and its blurry and noisy measurement together with elastic-net 
reconstructions for $\alpha=1\times10^{-4}$ and different values of $\beta$ and 
the $\ell^1$-reconstruction.}
\label{fig:blur}
\end{figure}

\section{Conclusion}
\label{sec:conclusion}

We analyzed the elastic-net regularization from an ``inverse
problem'' point of view. Using classical and modern techniques we
showed that elastic-net regularization combines the best of both
$\ell^2$- and $\ell^1$-regularization, i.e.~the good convergence
rate of $\ell^1$-regularization and modest constants in the error
estimates from $\ell^2$-regularization. Moreover, we also showed
that the \emph{a posteriori} parameter choice due to Morozov also
works for elastic-net regularization and leads to the same
convergence rates as our \emph{a priori} choice. Large parts of our
analysis were based on a linear coupling of the two regularization
parameters. However, Theorem~\ref{thm:consistency} indicates that
an asymptotic linear coupling of the parameters would suffice.
From Example~\ref{ex:dep_on_eta} one may conjecture that there is a
critical value of the coupling constant $\eta$ for all values greater 
than which the minimal-$\eta\norm{\cdot}_{\ell^1} +
\tfrac{1}{2}\norm{\cdot}_{\ell^2}^2$-solution coincides with the
minimal-$\norm{\cdot}_{\ell^1}$-solution. This would provide a
further justification for the elastic-net functional.

We have also developed two active set methods for minimizing the
elastic-net functional and numerically confirmed their excellent
performance. We may state that elastic-net is coequal to
classical $\ell^1$ minimization in terms of relative error, sparsity
and computation time for well conditioned problems and is favorably
for ill-conditioned problems.

\section*{Acknowledgements}
The authors would like to thank the referees for their comments.
Bangti Jin is supported by the Alexander von Humboldt foundation
through a postdoctoral researcher fellowship, and Dirk A. Lorenz 
by the German Science Foundation under grant LO 1436/2-1.

\bibliographystyle{plain}
\bibliography{literatur}

\begin{thebibliography}{10}

\bibitem{Attouch2006}
H.~Attouch, G.~Buttazo, and G.~Michaille.
\newblock {\em Variational Analysis in Sobolev and $BV$ Spaces: Applications to
  PDEs and Optimization}.
\newblock SIAM, Philadelphia, 2006.

\bibitem{Bonesky2009}
Thomas Bonesky.
\newblock Morozov's discrepancy principle and {T}ikhonov-type functionals.
\newblock {\em Inverse Problems}, 25(1):015015, 2009.

\bibitem{bredies2008harditer}
Kristian Bredies and Dirk~A. Lorenz.
\newblock Iterated hard shrinkage for minimization problems with sparsity
  constraints.
\newblock {\em SIAM Journal on Scientific Computing}, 30(2):657--683, 2008.

\bibitem{bredies2008softthresholding}
Kristian Bredies and Dirk~A. Lorenz.
\newblock Linear convergence of iterative soft-thresholding.
\newblock {\em Journal of Fourier Analysis and Applications},
  14(5--6):813--837, 2008.

\bibitem{burger2004convarreg}
Martin Burger and Stanley Osher.
\newblock Convergence rates of convex variational regularization.
\newblock {\em Inverse Problems}, 20(5):1411--1420, 2004.

\bibitem{Candes2006}
Emmanuel~J. Cand{\'e}s and Terrence~C. Tao.
\newblock Near optimal signal recovery from random projections: universal
  encoding strategies.
\newblock {\em IEEE Transactions on Information Theory}, 52(1):5406--5425,
  2006.

\bibitem{chen1998basispursuit}
Scott~Shaobing Chen, David~L. Donoho, and Michael~A. Saunders.
\newblock Atomic decomposition by basis pursuit.
\newblock {\em SIAM Journal on Scientific Computing}, 20(1):33--61, 1998.

\bibitem{chen2000semismooth}
Xiaojun Chen, Zuhair Nashed, and Liqun Qi.
\newblock Smoothing methods and semismooth methods for nondifferentiable
  operator equations.
\newblock {\em SIAM Journal on Numerical Analysis}, 38(4):1200--1216, 2000.

\bibitem{daubechies2003iteratethresh}
Ingrid Daubechies, Michel Defrise, and Christine De~Mol.
\newblock An iterative thresholding algorithm for linear inverse problems with
  a sparsity constraint.
\newblock {\em Communications in Pure and Applied Mathematics},
  57(11):1413--1457, 2004.

\bibitem{DeMol.etal:2009}
C.~De~Mol, E.~De~Vito, and L.~Rosasco.
\newblock Elastic-net regularization in learning theory.
\newblock {\em Journal of Complexity}, 25(2):201--230, 2009.

\bibitem{engl1996inverseproblems}
Heinz~W. Engl, Martin Hanke, and Andreas Neubauer.
\newblock {\em Regularization of Inverse Problems}, volume 375 of {\em
  Mathematics and its Applications}.
\newblock Kluwer Academic Publishers Group, Dordrecht, 2000.

\bibitem{figueiredo2003emalgorithm}
M{\'a}rio A.~T. Figueiredo and Robert~D. Nowak.
\newblock An {EM} algorithm for wavelet-based image restoration.
\newblock {\em IEEE Transactions on Image Processing}, 12(8):906--916, 2003.

\bibitem{figueiredo2007gradproj}
M{\'a}rio A.~T. Figueiredo, Robert~D. Nowak, and Stephen~J. Wright.
\newblock Gradient projection for sparse reconstruction: Applications to
  compressed sensing and other inverse problems.
\newblock {\em IEEE Journal of Selected Topics in Signal Processing},
  1(4):586--597, 2007.

\bibitem{grasmair2008sparseregularization}
Markus Grasmair, Markus Haltmeier, and Otmar Scherzer.
\newblock Sparse regularization with $\ell^q$ penalty term.
\newblock {\em Inverse Problems}, 24(5):055020 (13pp), 2008.

\bibitem{griesse2008ssnsparsity}
Roland Griesse and Dirk~A. Lorenz.
\newblock A semismooth {N}ewton method for {T}ikhonov functionals with sparsity
  constraints.
\newblock {\em Inverse Problems}, 24(3):035007 (19pp), 2008.

\bibitem{ElaineT.Hale.etal:2008}
Elaine~T. Hale, Wotao Yin, and Yin Zhang.
\newblock Fixed-point continuation for l1-minimization: Methodology and
  convergence.
\newblock {\em SIAM J. Optim.}, 19(3):1107--1130, 2008.

\bibitem{Hansen2007}
Per~Christian Hansen.
\newblock Regularization {T}ools {V}ersion 4.0 for {M}atlab 7.3.
\newblock {\em Numerical Algorithms}, 46:189--194, 2007.

\bibitem{hintermueller2003primaldualssn}
Michael Hinterm\"uller, Kazufumi Ito, and Karl Kunisch.
\newblock The primal-dual active set strategy as a semismooth {N}ewton method.
\newblock {\em SIAM Journal on Optimization}, 13(3):865--888, 2003.

\bibitem{Ito1992}
K.~Ito and K.~Kunisch.
\newblock On the choice of the regularization parameter in nonlinear inverse
  problems.
\newblock {\em SIAM Journal on Optimization}, 2(3):376--404, 1992.

\bibitem{Jin2009}
Bangti Jin and Jun Zou.
\newblock Iterative schemes for {M}orozov's discrepancy principle in
  optimizations arising from inverse problems.
\newblock {\em submitted}, 2009.

\bibitem{lee2006featuresignsearch}
Honglak Lee, Alexis Battle, Rajat Raina, and Andrew~Y. Ng.
\newblock Efficient sparse coding algorithms.
\newblock In B.~Sch\"{o}lkopf, J.~Platt, and T.~Hoffman, editors, {\em Advances
  in Neural Information Processing Systems 19}, pages 801--808. MIT Press,
  Cambridge, MA, 2007.

\bibitem{Levy1981}
S.~Levy and P.~Fullager.
\newblock Reconstruction of a sparse spike train from a portion of its spectrum
  and application to high-resolution deconvolution.
\newblock {\em Geophysics}, 46(9):1235--1243, 1981.

\bibitem{lorenz2008reglp}
Dirk~A. Lorenz.
\newblock Convergence rates and source conditions for {T}ikhonov regularization
  with sparsity constraints.
\newblock {\em Journal of Inverse and Ill-Posed Problems}, 16(5):463--478,
  2008.

\bibitem{Loris2009}
Ignace Loris.
\newblock On the performance of algorithms for the minimization of
  $\ell^1$-penalized functionals.
\newblock {\em Inverse Problems}, 25(3):035008, 2009.

\bibitem{Rockafellar1970}
R.~Tyrrell Rockafellar.
\newblock {\em Convex Analysis}.
\newblock Princeton University Press, Princeton, 1970.

\bibitem{Taylor1979}
H.~Taylor, S.~Bank, and J.~McCoy.
\newblock Deconvolution with the $\ell^1$ norm.
\newblock {\em Geophysics}, 44(1):39--52, 1979.

\bibitem{Tibshirani1996}
R.~Tibshirani.
\newblock Regression shrinkage and selection via the lasso.
\newblock {\em Journal of the Royal Statistical Society, Series B},
  58(1):267--288, 1996.

\bibitem{Ulbrich2002}
S.~Ulbrich.
\newblock Semismooth newton methods for operator equations in function spaces.
\newblock {\em SIAM Journal on Optimization}, 13(3):805--842, 2002.

\bibitem{Wright2009}
S.~J. Wright, R.~D. Nowak, and M.~A.~T. Figueiredo.
\newblock Sparse reconstruction by separable aproximation.
\newblock {\em IEEE Transactions on Signal Processing}, 57(7):2479--2493, 2009.

\bibitem{Zou.etal:2008}
Hui Zou and Trevor Hastie.
\newblock Regularization and variable selection via the elastic net.
\newblock {\em Journal of the Royal Statistical Society, Series B},
  67(2):301--320, 2005.

\end{thebibliography}

\end{document}